\documentclass{amsart}
\usepackage{amsmath}
\usepackage{color}
\usepackage{mathtools}

\newcommand{\ol}{\overline}

\newtheorem{theorem}{Theorem}[section]
\newtheorem{lemma}[theorem]{Lemma}
\newtheorem{proposition}[theorem]{Proposition}
\newtheorem{corollary}[theorem]{Corollary}

 \theoremstyle{definition}
\newtheorem{definition}[theorem]{Definition}

\theoremstyle{remark}

\numberwithin{equation}{section}

\begin{document}

\title[Geometric flow on affine manifolds]
{A geometric flow on noncompact affine Riemannian manifolds}

\author{Heming Jiao}
\address{School of Mathematics and Institute for Advanced Study in Mathematics, Harbin Institute of Technology,
         Harbin, Heilongjiang 150001, China}
\email{jiao@hit.edu.cn}

\author{Hanzhang Yin}
\address{School of Mathematics, Harbin Institute of Technology,
         Harbin, Heilongjiang 150001, China}
\email{21b912010@stu.hit.edu.cn}
\thanks{The first author is supported by the NSFC (Grant No. 12271126),
the Natural Science Foundation of Heilongjiang Province (Grant No. YQ2022A006),
and the Fundamental Research Funds for the Central Universities (Grant No. HIT.OCEF.2022030).}


\begin{abstract}

In this paper, we obtain the existence criteria for a geometic flow on noncompact affine Riemannian manifolds.
Our results can be regarded as a real version of Lee-Tam \cite{r4}.
As an application, we prove that a complete noncompact Hessian manifold with nonnegative Hessian sectional curvature
and bounded geometry
is diffeomorphic to $\mathbb{R}^n$ if its tangent bundle has maximal volume growth.


\end{abstract}

\maketitle

\section{Introduction}

Let $(M, \nabla, g_0)$ be a complete noncompact affine Riemannian manifold, with $\nabla$ a flat, torsion free connection
and $g_0$ a Riemannian metric on $M$. In this paper, we consider the geometric flow
\begin{equation}
\label{e1.2}
\left\{ \begin{aligned}
   &\frac{\partial }{\partial t}{{g}}=-\beta (g)  \\
      &g(0)={{g}_{0}}  \\
\end{aligned} \right.
\end{equation}
on $M$, where $\beta (g)$ is defined by
\[
{{\beta }_{ij}}(g)=-{{\partial }_{i}}{{\partial }_{j}}\log \det [g(t)]
\]
locally with respect to an affine coordinate system. When $g_0$ is Hessian, i.e., $g_0$ can be locally expressed
by $g_0=\nabla d\varphi$,
the evolved metrics $g(t)$ along the flow remain Hessian (seeing \cite{r12}) and $-\frac{\beta}{2}$ is called the second Koszul form
of $g(t)$ (see Definition \ref{Kos}).

Affine manifolds arise naturally in the the study of affine differential geometry (\cite{CY77, LYZ05, HW17}).
The Hessian metric, which is also called affine K\"{a}hler manifolds by Cheng-Yau \cite{r10}, is the Riemannian metric
compatible with the flat structure. Hessian metrics play a key role in various fields. They appear in the study of
special K\"{a}hler manifolds (\cite{F99, H99}) as well as in physics (\cite{SYZ96}), statistical manifolds and
information geometry (\cite{Am16, AN00, AJLS17}). The reader is referred to \cite{r11} for more details on
the geometry of Hessian manifolds.

Affine Riemannian manifolds are closely related to Hermitian manifolds. Indeed, the
tangent bundle of an affine Riemannian manifold is a Hermitian manifold, while the Hermitian metric
is K\"{a}hler if and only if the Riemannian metric $g_0$ is Hessian (see Section 2.2 for more details).
When $g_0$ is Hessian, \eqref{e1.2} may be regarded as a real analogue of K\"{a}hler-Ricci flow,
while \eqref{e1.2} may be considered as a real analogue of Chern-Ricci flow for general $g_0$. It is worth
mentioning that not all affine manifolds admit Hessian metrics. Indeed, an affine Hopf manifold dose not admit any
Hessian metric (c.f. Chapter 7 in \cite{r11}). In this paper, we are mainly concerned with the long time behaviour of the geometric flow \eqref{e1.2}
on noncompact affine Riemannian manifolds. The flow \eqref{e1.2} on compact Hessian manifolds was introduced and
studied by Mirghafouri-Malek \cite{r12} and further studied by Puechmorel-T{\^o} \cite{r13} in which \eqref{e1.2}
is called the Hesse-Koszul flow.

Now let's review some researches on Chern-Ricci flow and
K\"{a}hler-Ricci flow. The Ricci flow on Riemannian manifolds was introduced by
Hamilton \cite{Ha82}. Cao \cite{Cao85} first proved the convergence of the K\"{a}hler flow to the canonical K\"{a}hler-Einstein
metric on a compact K\"{a}hler manifold $M$ with its first Chern class $c_1 (M) < 0$ or $c_1 (M) = 0$. The K\"{a}hler-Ricci
flow on noncompact complete K\"{a}hler manifolds was first studied by Shi \cite{Shi89}.
The Chern-Ricci flow on a compact K\"{a}hler manifold was first studied by Gill \cite{Gill11} in a special setting
and was introduced and
generalized by Tosatti-Weinkove \cite{r3} later. Lee-Tam \cite{r4} considered the Chern-Ricci flow on a noncompact
Hermitian manifold. The reader is referred to a survey of Tosatti-Weinkove \cite{TW21} for more details on the Chern-Ricci flow.


Our main goal is to give the existence criteria for \eqref{e1.2} on noncompact affine Riemannian manifolds.
Let $S_A$ be the supremum of $S>0$ so that the flow \eqref{e1.2} has a solution $g(t)$ with initial data $g_0$ such that $g(t)$ is uniformly equivalent to $g_0$ in $M\times [0,S]$
and $S_B$ the supremum of $S>0$ such that there is a smooth bounded function $u$ satisfying
\begin{equation}
\label{e5.1}
g_0-S\beta(g_0)+\nabla du\geq \theta g_0
\end{equation}
for some $\theta>0$.
\begin{theorem}
\label{t1.2}
Let $(M, \nabla, g_0)$ be a complete noncompact affine Riemannian manifold. Assume that

\textbf{\emph{(i)}}
\begin{minipage}[t]{0.9\linewidth}
There exists a smooth real function $\rho$ which is uniformly equivalent to the distance function from a fixed point such that $|d\rho|_{g_0}$, $|\nabla d\rho|_{g_0}$ are uniformly bounded.
\end{minipage}

\textbf{\emph{(ii)}}
\begin{minipage}[t]{0.9\linewidth}
For any point $p \in M^n$, there exists a local coordinates $\{x^1,\ldots ,x^n\}$ around $p$ such that $(g_0)_{ij}(p)=\delta_{ij}$, $\partial_k(g_0)_{ij}(p)$ and $\partial_i(\partial_j (g_0)_{kl}-\partial_l (g_0)_{kj})(p)$ are uniformly bounded, $\partial_l \partial_k(g_0)_{ij}(p)$ are uniformly bounded from above and all the changes
of these local coordinate systems are affine.
\end{minipage}

Then $S_A=S_B$.
\end{theorem}
When $g_0$ is a Hessian metric, by Theorem \ref{l3.3}, we have
\begin{theorem}
\label{thm2}
Let $(M, \nabla, g_0)$ be a complete noncompact Hessian manifold. If for any point $p \in M^n$, there exists a local coordinates $\{x^1,\ldots ,x^n\}$ around $p$ such that $(g_0)_{ij}(p)=\delta_{ij}$, $|\partial_k(g_0)_{ij}|(p)\leq K$, the Hessian curvature $|Q_{ijkl}(p)|\leq K$ for some constant $K$ independent of $p$ and the
local affine coordinate systems. Then $S_A=S_B$.
\end{theorem}
A longstanding conjecture of Yau \cite{Yau89} states that
a complete noncompact K\"{a}hler manifold with positive holomorphic
bisectional curvature is biholomorphic to $\mathbb{C}^n$. It was proved by Liu \cite{Liu19} that a complete noncompact
K\"{a}hler manifold with nonnegative bisectional curvature and maximal volume growth is biholomorphic to
$\mathbb{C}^n$. The same result was proved by Lee-Tam \cite{r4} later using the Chern-Ricci flow.
In this paper, we prove similar conclusion for noncompact Hessian manifolds
as an application of Thereom \ref{thm2}:
\begin{theorem}
\label{uniformization}
Let $(M^n, g_0)$ be a complete noncompact connected Hessian manifold with nonnegative Hessian bisectional curvature
(see Definition \ref{bisectional}). Assume that
\begin{equation}
\label{MVG}
\mathrm{Vol}_{g^T} (B (p, r)) \geq C_0 r^{2n}, \ \forall r \in [0, \infty) \mbox{ for some } p \in TM
\end{equation}
and for some positive constant $C_0$. If for any point $p \in M^n$, there exists a local coordinates $\{x^1,\ldots ,x^n\}$ around $p$ such that $(g_0)_{ij}(p)=\delta_{ij}$, $|\partial_k(g_0)_{ij}(p)|\leq K$, the Hessian curvature $|Q_{ijkl}(p)|\leq K$ for some constant $K$ independent of $p$ and all the changes
of these local coordinate systems are affine. Then $M$ is diffeomorphic to $\mathbb{R}^n$.
\end{theorem}
In view of Theorem \ref{uniformization}, the current work can be also regarded as an attempt to understand real manifolds from their tangent bundles.
We remark that the tangent bundle of a smooth manifold (which is not necessary affine) being diffeomorphic to $\mathbb{R}^{2n}$
does not imply that the manifold is diffeomorphic to $\mathbb{R}^n$. A classical 3-dimensional counterexample is Whitehead's manifold $W$.
The tangent bundle of $W$ is diffeomorphic to $\mathbb{R}^6$, however, $W$ is topologically distinct from $\mathbb{R}^3$
(c.f. \cite{Rolfsen76}, Page 82).

In general, given a linear connection $\nabla$ on $M$, it was shown in \cite{Dombrowski62} that the tangent bundle $TM$ admits an almost complex
structure $J$ determined by the connection $\nabla$ and that $J$ is integrable if and only if $\nabla$ is flat. It may be of interest to
investigate geometrical properties of a real manifold via its tangent bundle.



The rest of this paper is organized as follows. In Section 2, we provide some preliminaries which may be used later.
In Section 3, the short time existence is considered. We establish the \emph{a priori} estimates up to $C^2$
for the flow \eqref{e1.2} in Section 4.
Theorem \ref{t1.2} is proved in Section 5. Section 6 is devoted to prove Theorem \ref{uniformization}.

\textbf{Acknowledgement.} The first author wish to thank Professor Ben Weinkove for suggesting him to read \cite{r13}
and to study the geometric flow on affine/Hessian manifolds while he was visiting the Department
of Mathematics at Northwestern University.

\section{Preliminaries}

In this section, we provide some basic results and preliminaries which may be used in the following sections.

First as the Chern-Ricci flow, \eqref{e1.2} can be rewritten as a scalar equation, which is a kind of parabolic Monge-Amp\`{e}re equations:
\begin{equation}
\label{e1.3}
\left\{ \begin{aligned}
   &\frac{\partial \varphi }{\partial t}=\log \frac{\det ({{g}_{0}}-t\beta ({{g}_{0}})+\nabla d\varphi )}{\det ({{g}_{0}})},  \\
   &{{g}_{0}}-t\beta ({{g}_{0}})+\nabla d\varphi >0,  \\
   &\varphi (0)=0.  \\
\end{aligned} \right.
\end{equation}
where $\varphi$ is an unknown function.
Indeed, given $\varphi (t)$ solving \eqref{e1.3}, we see $g(t):={g}_{0}-t\beta ({{g}_{0}})+\nabla d\varphi$ satisfies
\[\frac{\partial }{\partial t}g(t)=-\beta ({{g}_{0}})+\nabla d\log \frac{\det (g(t))}{\det ({{g}_{0}})}=-\beta ({{g}_{t}})\]
and $g(0)=g_0$. It follows that $g(t)$ is a solution to \eqref{e1.2}.

Conversely, if $g(t)$ is a solution to \eqref{e1.2}, we define $\varphi(x,t)$ by
\[\varphi (x,t)=\int_{0}^{t}{\log \frac{\det g(x,s)}{\det {{g}_{0}}(x)}}ds.\]
We find
\[\frac{\partial \varphi }{\partial t}=\log \frac{\det g(x,s)}{\det {{g}_{0}}(x)},\ \ \varphi (x,0)=0\]
By direct calculations, we obtain
\[\frac{\partial}{\partial t}(g(t)-g_0+t\beta(g_0)-\nabla d\varphi)=0\]
and
\[(g(t)-g_0+t\beta(g_0)-\nabla d\varphi)|_{t=0}=0.\]
Therefore, $g(t)=g_0-t\beta(g_0)+\nabla d\varphi>0$ for all $t\in [0,T)$ as required.

Now we provide some basic definitions and results on affine and Hessian manifolds. Most of them can be found in \cite{r11}.


\begin{definition} An affine manifold $(M,\nabla)$ is a differentiable manifold endowed with a flat connection $\nabla$.
\end{definition}
The following proposition can be derived easily from the definition of affine manifolds.
\begin{proposition}
\label{prop1}
\

\emph{(1)}
\begin{minipage}[t]{0.92\linewidth}
Suppose that $M$ admits a flat connection $\nabla$. Then there exist local coordinate systems on $M$ such that $\nabla_{\partial /\partial x^i} \partial /\partial x^j=0$. The changes between such coordinate systems are affine transformations.
\end{minipage}

\emph{(2)}
\begin{minipage}[t]{0.92\linewidth}
Conversely, if $M$ admits local coordinate systems such that the changes of the local coordinate systems are affine transformations, then there exists a flat connection $\nabla$ satisfying $\nabla_{\partial /\partial x^i} \partial /\partial x^j=0$ for all such local coordinate systems.
\end{minipage}
\end{proposition}
Let $(M, \nabla, g)$ be an affine manifold with a Riemannian metric $g$ and denote $\hat{\nabla}$ the Levi-Civita connection of
the Riemannian manifold $(M,g)$. Let $\gamma=\hat{\nabla}-\nabla$. Since $\nabla$ and $\hat{\nabla}$ are both torsion-free, we have
\[\gamma_X Y=\gamma_Y X\]
Moreover, the components $\gamma_{\;jk}^i$ of $\gamma$ with respect to affine coordinate systems coincide with the Christoffel symbols $\Gamma_{\;jk}^i$ of the Levi-Civita connection $\hat{\nabla}$.

The tensor $Q=\nabla \gamma$ is called the Hessian curvature tensor for $(g,\nabla)$.
A Riemannian metric $g$ on a affine manifold $(M,\nabla)$ is called Hessian if $g$ can be locally expressed by $g=\nabla d\varphi$. An affine manifold endowed with a Hessian metric is called a Hessian manifold.

For a flat connection $\nabla$, a local coordinate system $\{x^1,\ldots,x^n\}$ satisfying $\nabla_{\frac{\partial}{\partial x^i}} \frac{\partial}{\partial x^j}=0$ is called an affine coordinate system with respect to $\nabla$.
Let $(M,\nabla,g)$ be a Hessian manifold and $g$ can be locally expressed by
\[g_{ij}=\frac{\partial^2 \phi}{\partial x^i \partial x^j}\]
where $\{x^1,\ldots,x^n\}$ is an affine coordinate system with respect to $\nabla$. The following two propositions can be found in \cite{r11}.

\begin{proposition}
Let $(M,\nabla)$ be an affine manifold and $g$ a Riemannian metric on $M$. Then the following are equivalent:

\emph{(1)} $g$ is a Hessian metric

\emph{(2)} $(\nabla_X g)(Y,Z)=(\nabla_Y g)(X,Z)$

\emph{(3)} $\dfrac{\partial g_{ij}}{\partial x^k}=\dfrac{\partial g_{kj}}{\partial x^i}$

\emph{(4)} $g(\gamma_X Y,Z)=g(Y,\gamma_X Z)$

\emph{(5)} $\gamma_{ijk}=\gamma_{jik}$

\end{proposition}

\begin{proposition}\label{p2.4}
 Let $\hat{R}$ be the Riemannian curvature of a Hessian metric $g=\nabla d\phi$ and $Q=\nabla \gamma$ be the Hessian curvature tensor for $(g,\nabla)$. Then

\emph{(1)} $Q_{ijkl}=\dfrac{1}{2}\dfrac{\partial^4 \phi}{\partial x^i \partial x^j \partial x^k \partial x^l}-\dfrac{1}{2}g^{pq}\dfrac{\partial^3 \phi}{\partial x^i \partial x^k \partial x^p}\dfrac{\partial^3 \phi}{\partial x^j \partial x^l \partial x^q}$

\emph{(2)} $\hat{R}(X,Y)=-[\gamma_X,\gamma_Y],\;\;\;\hat{R}^i_{\;jkl}=\gamma^i_{\;lm}\gamma^m_{\;jk}-\gamma^i_{\;km}\gamma^m_{\;jl}.$

\emph{(3)}
$\hat{R}_{ijkl}=\dfrac{1}{2}(Q_{ijkl}-Q_{jikl})$

\end{proposition}

\begin{definition}
\label{Kos}
Let $(M,\nabla,g)$ be a Hessian manifold and
$v$ the volume element of $g$. The first Koszul form $\alpha$ and the second Koszul form $\kappa$ for $(\nabla,g)$ are defined by
\[\nabla_X v=\alpha(X)v\]
\[\kappa=\nabla \alpha\]
\end{definition}
It follows that
\[\alpha(X)={\rm Tr} {\gamma_X}\]
and
\[\alpha_i=\frac{1}{2}\frac{\partial \log \det[g_{pq}]}{\partial x^i}=\gamma_{\;ki}^k\]
\[\kappa_{ij}=\frac{\partial \alpha_i}{\partial x^j}=\frac{1}{2}\frac{\partial \log \det[g_{pq}]}{\partial x^i \partial x^j}\]
locally.


Now we introduce the relationship between affine structures and complex structures.
Let $(M,\nabla)$ be a flat manifold and $TM$ be the tangent bundle of $M$ with projection $\pi:TM\rightarrow M$. For an affine coordinate system $\{x^1,\ldots,x^n\}$ on $M$, we set
\begin{equation}
z^j=\xi^j+\sqrt{-1}\xi^{n+j}
\end{equation}
where $\xi^i=x^i\circ \pi$ and $\xi^{n+j}=dx^i$. Then $n$-tuples of functions given by $\{z^1,\ldots,z^n\}$ yield holomorphic coordinate systems on $TM$. We denote by $J_\nabla$ the complex structure tensor of the complex manifold $TM$. For a Riemannian metric $g$ on $M$ we put
\begin{equation}\label{e2.2}
g^T=\sum_{i,j=1}^{n}(g_{ij}\circ \pi)dz^id\overline{z}^j.
\end{equation}
Then $g^T$ is a Hermitian metric on the complex manifold $(TM,J_\nabla)$.
\begin{proposition} (\cite{r11}) \label{p2.6}
Let $(M,\nabla)$ be a flat manifold and $g$ a Riemannian metric on $M$. Then the following conditions are equivalent.

\noindent\emph{(1)} $g$ is a Hessian metric on $(M,\nabla)$.

\noindent\emph{(2)} $g^T$ is a K{\"a}hler metric on $(TM,J_\nabla)$.

\end{proposition}
Furthermore, we have
\begin{proposition} (\cite{r11}) \label{p2.7}
Let $(M, \nabla, g)$ be a Hessian manifold and $R^T$ be the Riemannian curvature tensor of the K{\"a}hler manifold $(TM,J,g^T)$. Then we have
\begin{equation}
\label{sec}
R_{i\bar{j}k\bar{l}}^T=-\frac{1}{2}Q_{ijkl}\circ \pi.
\end{equation}
Let $R_{i\bar{j}}^T$ be the Ricci tensor of the K{\"a}hler manifold $(TM,J,g^T)$. Then we have
\[R_{i\bar{j}}^T=\frac{1}{4}\beta_{ij}\circ \pi.\]
\end{proposition}
\begin{definition}
\label{bisectional}
Let $(M,g_0)$ be a Hessian manifold. Let $Q$ be the Hessian curvature tensor for $(g,\nabla)$. By \eqref{sec}, in order to compatible
to the study of complex geometric flows, $g_0$ is said to have
nonnegative Hessian sectional curvature if at any point and for any orthogonal frame, $Q_{iijj}\leq 0$.
\end{definition}

\section{Short time existence}

In this section, we consider the short time existence of the flow \eqref{e1.2}.
It suffices
to prove the short time existence of the parabolic Monge-Amp\`{e}re equation \eqref{e1.3}. We use similar idea of Chau-Tam \cite{r1} where
a modified parabolic complex Monge-Amp\`{e}re equation was studied. To begin our discussion, we introduce the concept of bounded geometry of various orders
which were defined for the K{\"a}hler case in \cite{r1}.
\begin{definition}\label{d3.1} Let $({{M}^{n}},g)$ be a complete affine Riemannian manifold. Let $k\ge 1$ be an integer and $0<\alpha<1$.$g$ is said to have bounded geometry of order $k+\alpha$ if there are positive numbers $r,\kappa_1,\kappa_2$ such that at every $p\in M$ there is a neighborhood $U_p$ of $p$, and local diffeomorphism $\xi_p$ from $D(r):=\{x\in \mathbb{R}^n||x|<1\}$ onto $U_p$ with $\xi_p(0)=p$ satisfying the following properties:\\
\hspace*{0.5cm}(i) the pull back metric $\xi^{*}_p(g)$ satisfies:
\[\kappa_1 g_e\le \xi^{*}_p(g)\le \kappa_2 g_e\]
\hspace*{1cm}where $g_e$ is the standard metric on $\mathbb{R}^n$;\\
\hspace*{0.4cm}(ii) the components $g_{ij}$ of $\xi^{*}_p(g)$ in the natural coordinate of $D(r)\subset \mathbb{R}^n$ are
\hspace*{1cm}uniformly bounded in the standard $C^{k+\alpha}$ norm in $D(r)$ independent of $p$.\\
$g$ is said to have bounded geometry of infinity order if instead of $(ii)$ we have for any $k$, the $k$-$th$ derivatives of $g_{ij}$ in $D(r)$ are bounded by a constant independent of $p$.
\end{definition}
The set $\mathcal{F}=\{(\xi_p,U_p),p\in M\}$ is called a family of \emph{quasi-coordinate neighborhoods}. Now we define the $C^{k+\alpha}$ norm
for $m$-forms on $M$ as follows, where $0 \leq m \leq n$ is a integer.
For any domain $\Omega$ in $\mathbb{R}^n$ and integer $k\geq0$ and $0<\alpha<1$, we denote the standard $C^{k+\alpha}$ norm for functions on $\Omega$ by $|| \cdot||_{\Omega,k+\alpha}$. If $T>0$ and $k$ is even let $|| \cdot||_{\Omega \times (0,T];k+\alpha,k/2+\alpha/2}$ be the standard parabolic $C^{k+\alpha,k/2+\alpha/2}$ norm for functions on $\Omega \times (0,T]$.
The reader is referred to \cite{r2} for details.

We define the $C^{k+\alpha}$ norm for a smooth $m$-form $f$ on $M$ by
\begin{equation}
\|f\|_{m,k,\alpha}=\underset{p\in M}{\sup} \underset{I}{\max}||(\xi^{*}_{p}f)_I||_{D(r),k+\alpha}.
\end{equation}
where $I$ represents a multi-index and the $(\xi^{*}_{p}f)_I$'s are the local components of the form $\xi^{*}_{p}f$. Similarly, if $I>0$ and $k$ is even we define the $C^{k+\alpha,k/2+\alpha/2(M\times [0,T))}$ norm for a smooth time dependent $m$-form $f$ on $M\times [0,T)$ by
\begin{equation}
\|f\|_{m,k,\alpha}=\underset{p\in M}{\sup} \underset{I}{\max}||(\xi^{*}_{p}f)_I||_{D(r)\times [0,T);k+\alpha,k/2+\alpha/2}.
\end{equation}
Denote $C_m^{k+\alpha} (M)$ ($C_m^{k+\alpha, k/2+\alpha/2} (M \times [0,T))$) to be the space of smooth (time dependent) $m$-forms with
the norm $\|\cdot\|_{m,k,\alpha}$. We find that they are Banach spaces and
similar to Lemma 5.1 of \cite{r1}, independent of the choice of $\mathcal{F}$.

We introduce the following conditions on a complete noncompact affine Riemannian manifold $(M^n,g_0)$ as \cite{r4}:

\textbf{(a1)}
\begin{minipage}[t]{0.9\linewidth}
There exists a smooth real function $\rho$ which is uniformly equivalent to the distance function from a fixed point such that $|d\rho|_{g_0}$, $|\nabla d\rho|_{g_0}$ are uniformly bounded.
\end{minipage}

\textbf{(a2)}
\begin{minipage}[t]{0.9\linewidth}
There is a smooth bounded function $u$ and positive constants $S,\theta$ such that
\[g_0-S\beta(g_0)+\nabla du\geq \theta g_0\]
\end{minipage}

\textbf{(a3)}
\begin{minipage}[t]{0.9\linewidth}
For any point $p \in M^n$, there exists a local coordinates $\{x^1,\ldots ,x^n\}$ around $p$ such that $(g_0)_{ij}(p)=\delta_{ij}$, $\partial_k(g_0)_{ij}(p)$ and $\partial_i(\partial_j (g_0)_{kl}-\partial_l (g_0)_{kj})(p)$ are uniformly bounded, $\partial_l \partial_k(g_0)_{ij}(p)$ are uniformly bounded from above and all the changes
of these local coordinate systems are affine transformations. (In the following of the paper, we denote the uniformly bounded constant by $K>0$)
\end{minipage}
~\\

Let us recall a well-known result, of Shi \cite{shi97}.
\begin{lemma}\label{l3.3} Suppose $(M,g_{ij}(x))$ is an n-dimensional complete noncompact Riemannian manifold with its Ricci curvature bounded from below:
\[R_{ij}(x)\geq -k_0g_{ij}(x),\;\;\;\;\forall x\in M,\]
where $0\leq k_0<+\infty$ is a constant. Then there exists a constant $0<C_2<+\infty$ depending only on $n$ and $k_0$ such that for any fixed point $x_0 \in M$, there exists a smooth function $\rho(x)\in C^{\infty}(M)$ such that
\[\begin{aligned}
&C_2^{-1}[1+\gamma(x,x_0)]\leq \rho(x)\leq C_2[1+\gamma(x,x_0)]\\
&|\nabla \rho(x)|\leq C_2,\;\;\;\;\forall x\in M\\
&|\Delta \rho(x)|\leq C_2,\;\;\;\;\forall x\in M
\end{aligned}\]
where $\gamma$ is the distance function.
\end{lemma}
By a general implicit function Theorem, we get the following lemma:
\begin{lemma}\label{l3.4} Let $(M^n,g_0)$ be a smooth complete non-compact affine Riemannian manifold of bounded geometry of order $2+\alpha$ such that\\
\hspace*{0.5cm}\emph{(i)} $\beta(g_0) \in C^{\alpha,\alpha /2}(M\times [0,T))$;\\
\hspace*{0.4cm}\emph{(ii)} there exists $v_0\in C^{2+\alpha,1+\alpha /2}(M\times [0,T))$ such that\\
\[w_0:=\frac{\partial v_0}{\partial t}-\log \frac{\det ({{g}_{0}}-t\beta ({{g}_{0}})+\nabla d v_0 )}{\det ({{g}_{0}})}\]
\hspace*{0.8cm} satisfies $w_0(x,0)=0$.\\
Then there exists $0<T'\leq T$ such that \eqref{e1.3} has a solution $v\in C^{2+\alpha/2,1+\alpha /4}(M\times [0,T'))$ so that ${{g}_{0}}-t\beta ({{g}_{0}})+\nabla d v$ is uniformly equivalent to $g_0$ in $M\times [0,T')$.

\end{lemma}
\begin{proof}For convenience, we set $\sigma:={{g}_{0}}-t\beta (g_0)$. We find $\sigma$ satisfies (i) instead of $\beta(g_0)$
provided $t$ is sufficiently small and we can choose a smaller $T$ such that $\sigma$ is uniformly equivalent to $g_0$ in $M\times [0,T]$.
Thus, there exists a positive constant $C_1>0$ such that $C_1 g_0\geq\sigma\geq C^{-1}_1 g_0$ in $M\times [0,T]$.
From now on, we denote the spaces $C^{k+\beta,k/2+\beta /2}(M\times [0,T])$ and $C^{k+\beta}(M)$ simply by $C^{k+\beta,k/2+\beta /2}$ and $C^{k+\beta}$ for
convenience.
We define
\[\mathcal{B}=\{v\in C^{2+\alpha/2,1+\alpha/4}|\; \|v\|_{2+\alpha/2,1+\alpha/4}<\delta,v(x,0)=0\}\]
and
\[F:\mathcal{B}\rightarrow C^{\alpha/2,\alpha/4}\]
by
\[F(v)=\frac{\partial v}{\partial t}-\log \frac{\det (\sigma+\nabla d v )}{\det ({{g}_{0}})}\]
It is easy to find the Fr\'{e}chet derivative $DF_v$ of the map $F$ at any $v\in \mathcal{B}$ is given by
\[DF_{v}(\phi)=\frac{\partial\phi}{\partial t}-(^{v}\sigma)^{ij}\phi_{ij}\]
where $(^{v}\sigma)^{ij}$ is the inverse of $(^{v}\sigma)_{ij}:=\sigma_{ij}+v_{ij}$.

\underline{{\textbf {Claim}}} $DF_v$ is a bijection from the Banach space
\[\mathcal{B}_1=\{\phi \in C^{2+\alpha/2,1+\alpha/4}|\phi(x,0)=0\}\]
onto the space $C^{\alpha/2,\alpha/4}$.

Note that $(M^n,g_0)$ has bounded geometry of $2+\alpha$. By Lemma \ref{l3.3}, we find $(M^n,g_0)$ satisfies \textbf{(a1)}.
Let $\rho$ be a smooth function on $M$ equivalent to the distance function with respect to $g_0$ from a fixed point as in \textbf{(a1)}. Since
$(^{v}\sigma)$ are uniformly equivalent to $g_0$ and there is a constant $C_2$ such that
\[|(^{v}\sigma)_{ij}\rho_{ij}|<C_2\rho\]
in $M\times [0,T]$, we have $\phi\in \mathcal{B}_1$ and for any $\epsilon>0$,
\[\frac{\partial}{\partial t}(\phi+\epsilon e^{C_2 t}\rho)-(^{v}\sigma)_{ij}(\phi_{ij}+\epsilon e^{C_2 t}\rho_{ij})>0.\]
On the other hand, the minimum of $\phi+\epsilon e^{C_2 t}\rho$ is attained at a point in a compact subset of $M\times [0,T]$.
Thus, by the maximum principle, we have $\phi+\epsilon e^{C_2 t}\rho\geq 0$ since $\phi(x,0)=0$.
Letting $\epsilon\rightarrow 0$, we get $\phi \geq 0$ on $M\times [0,T]$.
Similarly, we can prove that $\phi\leq 0$. Hence $\phi=0$ and $DF_v$ is injective.

Now let $\Omega_l$ be a bounded domains with smooth boundary which exhaust $M$ and $^l \phi$ be the solution to
\[
\left\{ \begin{aligned}
   &DF_v(^l \phi) = w  \mbox{ in $\Omega_l\times [0,T]$}\\
      &^l \phi=0 \mbox{ on $\Omega_l \times \{0\} \cup \partial\Omega_l\times [0,T]$}, \\
\end{aligned} \right.
\]
where $w\in
C^{\alpha/2,\alpha/4}(M\times [0,T])$. Therefore, $DF_v(^l\phi+C_3 t)>0$, where $C_3$
is a positive constant chosen such that $C_3>\sup_{M\times [0,T]}|w|$. By the maximum principle, we have $^l\phi\geq -C_3$.
Similarly, we can prove $^l\phi\leq -C_3$. Hence the sequence $|^l\phi|$ is uniformly bounded by the constant $C_3$.

For any $p\in M$, let $(\xi_p,U_p)$ and $\xi_p:D(r)\rightarrow U_p$ be as in Definition \ref{d3.1}.
For simplicity, we also denote the pull back of $^l\phi$ by $^l\phi$. We find the pull back of $^l\phi$ satisfies
\[\frac{\partial ^l\phi}{\partial t}-(^{v}\sigma)^{ij}(^l\phi)_{ij}=w\]
in $D(r)$, where $\{(^{v}\sigma)^{ij}\}$ is uniformly equivalent
to the standard Euclidean metric. Furthermore, $(^{v}\sigma)^{ij}$ are uniformly bounded in the
standard $C^{\alpha,\alpha/2}$ norm on $D(r)\times [0,T]$. By the Schauder estimates we have
\[\|^l\phi\|_{D(r/2)\times [0,T],2+\alpha/2,1+\alpha/4}\leq C_4\]
for some constant $C_4$ independent of $p$ for sufficiently large $l$, Now a standard diagonalizing subsequence argument produces a $\phi\in C^{2+\alpha/2,1+\alpha/4}$ such that $DF_v(\phi)=w$. Since $DF_v$ is surjective and the claim is established.

Now let $v_0 \in C^{\alpha/2,\alpha/4}(M\times [0,T))$ be the function in (ii) and $w_0=F(v_0)$.
We have $w_0 (x,0)=0$. By the inverse function theorem of Banach space, there exists $\epsilon>0$ such that if $\|w-w_0\|_{\alpha/2,\alpha/4}<\epsilon$,
the equation $F(v)=w$ admits a solution $v\in C^{\alpha/2,\alpha/4}$.

For any $0<\tau<1$,we define
\[w_\tau (x,t)=\left\{ \begin{aligned}
&0 &\;\;\;t\leq \tau\\
w_0(x,\;&t-\tau)&\;\;\;\tau<t<1\\
\end{aligned} \right.\]
Similar to Claim 2 in the proof of Lemma 2.2 in \cite{r1}, we have that $\|w_\tau -w_0\|_{\alpha/2,\alpha/4}<\epsilon$ for sufficiently small $\tau>0$.
Then by the inverse function theorem we have $F(v)=0$ on $M\times [0,\tau]$ for some $\tau>0$. Hence the function $v$ solves \eqref{e1.3}. The fact that $v$ is smooth follows from a standard bootstrapping argument applied to \eqref{e1.3}. The proof is similar to that of Lemma \ref{l4.6} by differentiating the equation
with respect to $x^\gamma$ and using the parabolic Schauder theory. See also Proposition 2.1 of \cite{r1}.
\end{proof}
We are ready to prove the following short time existence.
\begin{lemma}\label{l3.5} Let $(M^n,g_0)$ be a smooth complete non-compact affine Riemannian manifold with bounded geometry of order $2+\alpha$ and $\beta(g_0)$ is in $C^{2+\alpha}(M)$ for some $\alpha>0$. Then there exists $T>0$ such that \eqref{e1.3} has a smooth solution $g(t)$ on $M\times [0,T)$ such that for each $t$, $g(t)$ is equivalent to $g_0$.
\end{lemma}
\begin{proof}Let
\[v_0=t\log \frac{\det ({{g}_{0}}-t\beta ({{g}_{0}}) )}{\det ({{g}_{0}})}.\]
Thus, $v_0$ satisfies condition (iii) in Lemma \ref{l3.4} on $M\times [0,T)$ for some $T>0$. Consequently,
there exists $0<T'\leq T$ such that \eqref{e1.3} has a smooth solution $v\in C^{2+\alpha/2,1+\alpha /4}(M\times [0,T'))$ so that $g(t)={{g}_{0}}-t\beta ({{g}_{0}})+\nabla d v$ is uniformly equivalent to $g_0$ in $M\times [0,T')$.
\end{proof}

\section{A priori estimates for the geometric flow}

Let $(M^n,g_0)$ be an affine Riemannian manifold and $g(t)$ a solution of \eqref{e1.2} with initial data $g(0)=g_0$ on $M\times [0,S]$ for some $S>0$.
Let $\varphi$ be a solution of the equation \eqref{e1.3} corresponding to $g(t)$, namely,
\[\varphi (x,t)=\int_{0}^{t}{\log \frac{\det g(x,s)}{\det {{g}_{0}}(x)}}ds\]
and $g(t)=g_0 -t\beta(g_0)+\nabla d\varphi$. In this section, we establish the \emph{a priori} estimates for \eqref{e1.3}.
\begin{lemma}\label{l4.1} Let $\Psi=t\dot{\varphi}-\varphi-nt$ and $\Lambda=(S_1 -t)\dot{\varphi}+\varphi+nt$ where $\dot{\varphi}=\frac{\partial}{\partial t}\varphi$ and $S_1$ is
an arbitrary given number. Then we have
\[(\frac{\partial}{\partial t}-L_g)\Psi=-\emph{tr}_g g_0\]
and
\[(\frac{\partial}{\partial t}-L_g)\Lambda=-S_1 \emph{tr}_g(\beta(g_0))+\emph{tr}_g g_0\]
where $L_g(f):=\emph{tr}_g \nabla df$.
\end{lemma}
\begin{proof}
Differentiating \eqref{e1.3} with respect to $t$, we have
\begin{equation}
\label{e4.1}
\begin{aligned}
\frac{\partial}{\partial t}\dot{\varphi}&=\frac{\partial}{\partial t}\log \frac{\det (g(t))}{\det ({{g}_{0}})}\\
&=g^{ij}\frac{\partial}{\partial t}g_{ij}\\
&=-{\rm tr}_{g}{\beta(g)}.
\end{aligned}
\end{equation}
By straightforward calculations, we have
\[\begin{aligned}
L_g \dot{\varphi}&={\rm tr}_{g}(\frac{\partial}{\partial t}g-\frac{\partial}{\partial t}({{g}_{0}}-t\beta ({{g}_{0}}))\\
&=-{\rm tr}_{g}{\beta(g)}+{\rm tr}_{g}\beta ({{g}_{0}}).
\end{aligned}\]
It follows that
\begin{equation}
\label{e4.2}
(\frac{\partial}{\partial t}-L_g)t\dot{\varphi}=\dot{\varphi}-t{\rm tr}_{g}\beta ({{g}_{0}}).
\end{equation}
By \eqref{e4.2}, we get
\[\begin{aligned}
(\frac{\partial}{\partial t}-L_g)\Psi&=(\frac{\partial}{\partial t}-L_g)\dot{\varphi}-\dot{\varphi}+{\rm tr}_{g}\nabla d\varphi-n \\
&=-{\rm tr}_{g}(g+t\beta(g_0)-\nabla d\varphi) \\
&=-{\rm tr}_{g}g_0.
\end{aligned}\]
By \eqref{e4.1}, we have
\[\begin{aligned}
(\frac{\partial}{\partial t}-L_g)\Lambda&=-(\frac{\partial}{\partial t}-L_g)\Psi+S_1(\frac{\partial}{\partial t}-L_g)\dot{\varphi} \\
&=-S_1 {\rm tr}_g(\beta(g_0))+{\rm tr}_g g_0.
\end{aligned}\]

\end{proof}

Next, we compute $(\frac{\partial}{\partial t}-L_g){\rm tr}_{g_0}g$ which can be viewed as a real version of the Chern-Ricci flow, as discussed by Tossati-Wienkove \cite{r3}.
\begin{lemma}\label{l4.2}
Let $\Upsilon={\rm tr}_{g_0}g$ and $\Theta={\rm tr}_g g_0$. We have
\[(\frac{\partial}{\partial t}-L_g)\log \Upsilon={\rm \uppercase\expandafter{\romannumeral1}}+{\rm \uppercase\expandafter{\romannumeral2}}\]
where
\[\begin{aligned}
{\rm \uppercase\expandafter{\romannumeral1}}&=\frac{1}{{\rm tr}_{g_0}g}\biggl(-g^{jp}g^{qi}g_0^{kl}\partial_k g_{ij}\partial_l g_{pq}+\frac{1}{{\rm tr}_{g_0}g}g^{lk}\partial_k {\rm tr}_{g_0}g\partial_l {\rm tr}_{g_0}g\\
&\;\;\;\;+2g_0^{li}g^{jp}g^{qk}g_0^{rl'}(\partial_k (g_0)_{l'i})g_{jr}\partial_l g_{pq}+2g^{qk}\partial_k g_0^{pl}(T_0)_{lq}^p\biggl)
\end{aligned}\]
\[\begin{aligned}
{\rm \uppercase\expandafter{\romannumeral2}}=-\frac{1}{{\rm tr}_{g_0}g}\biggl(g^{ij}g_0^{kl}(\partial_l(T_0)_{ik}^j+\partial_i(T_0)_{jl}^k)+g^{ij}(\partial_i\partial_j g_0^{kl})g_{kl}\biggl)
\end{aligned}.\]
Moreover we have
\[{\rm \uppercase\expandafter{\romannumeral1}}\leq \frac{1}{{\rm tr}_{g_0}g}\biggl(-2g_0^{li}g^{qk}(T_0)_{ik}^l\frac{\partial_q {\rm tr}_{g_0}g}{{\rm tr}_{g_0}g}+2g^{qk}\partial_k g_0^{pl}(T_0)_{lq}^p+g_0^{li}g^{jp}g^{qk}C_{ijk}C_{lpq}\biggl),\]
where
\[(T_0)_{jl}^k:=\partial_j (g_0)_{kl}-\partial_l (g_0)_{kj}\]
\[C_{ijk}=-g_0^{rl'}(\partial_k (g_0)_{l'i})g_{jr}.\]
\end{lemma}
\begin{proof}
We calculate at a fix point $x_0 \in M$. Choosing an affine coordinate system around $x_0$, we have
\begin{equation}
\begin{aligned}
\label{e4.3}
L_g\Upsilon&=g^{ij}\partial_i\partial_j(g_0^{kl}g_{kl})\\
&=g^{ij}g_0^{kl}\partial_i\partial_jg_{kl}+2g^{ij}(\partial_i g_0^{kl})(\partial_j g_{kl})+g^{ij}(\partial_i\partial_j g_0^{kl})g_{kl}.
\end{aligned}
\end{equation}
For convenience, we define
\[T_{jl}^k:=\partial_j g_{kl}-\partial_l g_{kj},\]
\[(T_0)_{jl}^k:=\partial_j (g_0)_{kl}-\partial_l (g_0)_{kj}.\]
By \eqref{e1.2}, we obtain
\[\begin{aligned}
\frac{\partial}{\partial t}{T_{jl}^k}&=\frac{\partial}{\partial t}(\partial_j g_{kl}-\partial_l g_{kj})\\
&=\frac{\partial^3\log(\det g)}{\partial x^k \partial x^l \partial x^j}-\frac{\partial^3\log(\det g)}{\partial x^k \partial x^j \partial x^l}\\
&=0.
\end{aligned}\]
It follows that $T_{jl}^k=(T_0)_{jl}^k$ for any $j,k,l$.
Hence
\begin{equation}
\begin{aligned}
\label{e4.4}
g^{ij}g_0^{kl}\partial_i\partial_jg_{kl}&=g^{ij}g_0^{kl}\partial_i(\partial_l g_{kj}+(T_0)_{jl}^k)\\
&=g^{ij}g_0^{kl}\partial_k\partial_l g_{ij}+g^{ij}g_0^{kl}(\partial_l(T_0)_{ik}^j+\partial_i(T_0)_{jl}^k).
\end{aligned}
\end{equation}
By straightforward calculations, we get
\begin{equation}
\label{e4.5}
\begin{aligned}
2g^{ij}(\partial_i g_0^{kl})(\partial_j g_{kl})=-2g_0^{li}g^{jp}g^{qk}g_0^{rl'}(\partial_k (g_0)_{l'i})g_{jr}\partial_l g_{pq}-2g^{qk}\partial_k g_0^{pl}(T_0)_{lq}^p.
\end{aligned}
\end{equation}
Combining \eqref{e4.3}, \eqref{e4.4} and \eqref{e4.5}, we have
\begin{equation}
\label{e4.6}
\begin{aligned}
L_g\Upsilon=&g^{ij}g_0^{kl}\partial_k\partial_l g_{ij}+g^{ij}g_0^{kl}(\partial_l(T_0)_{ik}^j+\partial_i(T_0)_{jl}^k)\\
&-2g^{qk}\partial_k g_0^{pl}(T_0)_{lq}^p-2g_0^{li}g^{jp}g^{qk}g_0^{rl'}(\partial_k (g_0)_{l'i})g_{jr}\partial_l g_{pq}\\
&+g^{ij}(\partial_i\partial_j g_0^{kl})g_{kl}.
\end{aligned}
\end{equation}
On the other hand,
\begin{equation}
\label{e4.7}
\begin{aligned}
\frac{\partial}{\partial t}\Upsilon&=g_0^{lk}\partial_k \partial_l \log\det(g)\\
&=g^{ji}g_0^{kl}\partial_k \partial_l g_{ji}-g^{jp}g^{qi}g_0^{kl}\partial_k g_{ij}\partial_l g_{pq}.
\end{aligned}
\end{equation}
By \eqref{e4.6} and \eqref{e4.7}, we have
\begin{equation}
\label{e4.8}
\begin{aligned}
(\frac{\partial}{\partial t}-L_g)\log\Upsilon&=\frac{1}{{\rm tr}_{g_0}g}\biggl(-g^{jp}g^{qi}g_0^{kl}\partial_k g_{ij}\partial_l g_{pq}+\frac{1}{{\rm tr}_{g_0}g}g^{lk}\partial_k {\rm tr}_{g_0}g\partial_l {\rm tr}_{g_0}g\\
&\;\;\;\;+2g_0^{li}g^{jp}g^{qk}g_0^{rl'}(\partial_k (g_0)_{l'i})g_{jr}\partial_l g_{pq}+2g^{qk}\partial_k g_0^{pl}(T_0)_{lq}^p\biggl)\\
&\;\;\;\;-\frac{1}{{\rm tr}_{g_0}g}\biggl(g^{ij}g_0^{kl}(\partial_l(T_0)_{ik}^j+\partial_i(T_0)_{jl}^k)+g^{ij}(\partial_i\partial_j g_0^{kl})g_{kl}\biggl)\\
&={\rm \uppercase\expandafter{\romannumeral1}}+{\rm \uppercase\expandafter{\romannumeral2}}.
\end{aligned}
\end{equation}

We now establish the estimates for ${\rm \uppercase\expandafter{\romannumeral1}}$. Note that
\[K=g_0^{li}g^{jp}g^{qk}B_{ijk}B_{lpq}\geq 0,\]
where
\[B_{ijk}=\partial_i g_{kj}-g_{ij}\frac{\partial_k {\rm tr}_{g_0}g}{{\rm tr}_{g_0}g}+C_{ijk},\]
and $C_{ijk}$ will be specified later.
Indeed, since the inverse of $g_0$ and $g$ are positive defined, they can be written as $g_0^{li}=M^{la}M^{ia}, g^{jp}=N^{jb}N^{pb}$
and we have $K=\underset{a,b,c}{\sum}(M^{ia}N^{jb}N^{kc}B_{ijk})^2\geq0$.
Next, we find
\[\begin{aligned}
K=&g_0^{li}g^{jp}g^{qk}\partial_i g_{kj}\partial_l g_{pq}+\frac{1}{{\rm tr}_{g_0}g}g^{qk}\partial_k {\rm tr}_{g_0}g\partial_q {\rm tr}_{g_0}g\\
&-2g_0^{li}g^{qk}\partial_i g_{kl}\frac{\partial_q {\rm tr}_{g_0}g}{{\rm tr}_{g_0}g}-2g_0^{li}g^{qk}C_{ijk}\frac{\partial_q {\rm tr}_{g_0}g}{{\rm tr}_{g_0}g}\\
&+g_0^{li}g^{jp}g^{qk}C_{ijk}\partial_l g_{pq}+g_0^{li}g^{jp}g^{qk}C_{ijk}C_{lpq}.
\end{aligned}\]
Using the identity
\[\begin{aligned}
g_0^{li}\partial_i g_{kl}&=g_0^{li}\partial_k g_{il}+g_0^{li}(T_0)_{ik}^l\\
&=\partial_k {\rm tr}_{g_0}g-\partial_k(g_0^{li})g_{il}+g_0^{li}(T_0)_{ik}^l,
\end{aligned}\]
in the third term we get
\[\begin{aligned}
K=&g_0^{li}g^{jp}g^{qk}\partial_i g_{kj}\partial_l g_{pq}-\frac{1}{{\rm tr}_{g_0}g}g^{qk}\partial_k {\rm tr}_{g_0}g\partial_q {\rm tr}_{g_0}g\\
&-2g_0^{li}g^{qk}(T_0)_{ik}^l\frac{\partial_q {\rm tr}_{g_0}g}{{\rm tr}_{g_0}g}-2g^{qk}(C_{ijk}g_0^{ji}-(\partial_k g_0^{li})g_{il})\frac{\partial_q {\rm tr}_{g_0}g}{{\rm tr}_{g_0}g}\\
&+g_0^{li}g^{jp}g^{qk}C_{ijk}\partial_l g_{pq}+g_0^{li}g^{jp}g^{qk}C_{ijk}C_{lpq},
\end{aligned}\]
Comparing the expression of $K$ and ${\rm \uppercase\expandafter{\romannumeral1}}$ we find
\[\begin{aligned}
{\rm \uppercase\expandafter{\romannumeral1}}=&\frac{1}{{\rm tr}_{g_0}g}\biggl(-K-2g_0^{li}g^{qk}(T_0)_{ik}^l\frac{\partial_q {\rm tr}_{g_0}g}{{\rm tr}_{g_0}g}\\
&-2g_0^{ji}g^{qk}[C_{ijk}+g_0^{ll'}(\partial_k (g_0)_{l'i})g_{jl}]\frac{\partial_q {\rm tr}_{g_0}g}{{\rm tr}_{g_0}g}\\
&+2g_0^{li}g^{jp}g^{qk}[C_{ijk}+g_0^{rl'}(\partial_k (g_0)_{l'i})g_{jr}]\partial_l g_{pq}\\
&+2g^{qk}\partial_k g_0^{pl}(T_0)_{lq}^p+g_0^{li}g^{jp}g^{qk}C_{ijk}C_{lpq} \biggl).
\end{aligned}\]
Let
\[C_{ijk}=-g_0^{rl'}(\partial_k (g_0)_{l'i})g_{jr}.\]
We have
\begin{equation}
\label{e4.9}
\begin{aligned}
{\rm \uppercase\expandafter{\romannumeral1}}\leq \frac{1}{{\rm tr}_{g_0}g}\biggl(-2g_0^{li}g^{qk}(T_0)_{ik}^l\frac{\partial_q {\rm tr}_{g_0}g}{{\rm tr}_{g_0}g}+2g^{qk}\partial_k g_0^{pl}(T_0)_{lq}^p+g_0^{li}g^{jp}g^{qk}C_{ijk}C_{lpq}\biggl)
\end{aligned}
\end{equation}
as required.

\end{proof}
Analogous to Lemma 3.3, 3.4 and 3.5 in \cite{r4} where the Chern-Ricci flow on noncompact Hermitian manifolds was studied, we need to prove
the following lemmas in our settings.
\begin{lemma}\label{l4.3} Let $(M^n,g_0)$ be a complete noncompact affine Riemannian manifold satisfying condition \emph{\textbf{(a1)}}: There exists a smooth real function $\rho$ which is uniformly equivalent to the distance function from a fixed point such that $|d\rho|_{g_0}$, $|\nabla d\rho|_{g_0}$ are uniformly bounded. Suppose $g(t)$ is a solution to the geometric flow
\eqref{e1.2} on $M\times [0,S)$. Assume that for any $0<S_1<S$, there is $C>0$ such that
\[C^{-1}g(t)\leq g_0\leq Cg(t)\]
for any $0<t\leq S_1$. Let $f$ be a smooth function on $M\times [0,S)$ which is bounded from above such that
\[(\frac{\partial}{\partial t}-L_g)f\leq 0\]
on $\{f>0\}$. Suppose $f\leq 0$ at $t=0$, then $f\leq 0$ on $M\times [0,S)$.
\end{lemma}
\begin{proof}We may assume that $g(t)$ is a solution to \eqref{e1.2} on $M\times [0,S)$ and is uniformly equivalent to $g_0$ on $M\times [0,S)$. Let $r(\cdot)$ be the distance
function from a fixed point $x_0$.  We may assume that there is a constant $C_1>0$ such that
\[C_1^{-1}(1+r(x))\leq \rho(x)\leq C_1(1+r(x)).\]
and $\rho\geq 1$. Since $g(t)$ is uniformly equivalent to $g_0$, there is a constant $C_2>0$ such that $L_g \rho\leq C_2$. Hence
\[(\frac{\partial}{\partial t}-L_g)(e^{2C_2t}\rho)\geq e^{2C_2t}(2C_2\rho-C_2)\geq C_2e^{2C_2t}\rho\]
For any $\epsilon>0$, if the $\sup_{M\times [0,T]}(f-\epsilon e^{2C_2 t}\rho)>0$, then there is a point $(x_0,t_0)$ with $t_0>0$ such that $f-\epsilon e^{2C_2 t}\rho\leq 0$ on $M\times[0,t_0]$ and $f-\epsilon e^{2C_2 t}\rho=0$ at $(x_0,t_0)$. Thus, $f(x_0,t_0)>0$, and at $(x_0,t_0)$ we get
\[0\leq (\frac{\partial}{\partial t}-L_g)(f-\epsilon e^{2C_2 t}\rho)<0,\]
which is impossible. Since $\epsilon$ is arbitrary, the lemma is proved.

\end{proof}

Assume that $g_0$ satisfies \textbf{(a1)}, \textbf{(a2)} and \textbf{(a3)} (in Page 7).
Let $u$ and $S$ be defined in \textbf{(a2)}. Suppose that $g(t)$ is a solution to the geometric flow \eqref{e1.2} on $M\times[0,S_1]$
by choosing $S_1<S$ such that $g(t)$ is uniformly equivalent to $g_0$ on $M\times[0,S_1]$.
Let $\varphi$ be the corresponding solution to \eqref{e1.3}. 
\begin{lemma}\label{l4.4} Suppose that $g_0$ satisfies \emph{\textbf{(a1)}}, \emph{\textbf{(a2)}} and \emph{\textbf{(a3)}}. Then there is a constant $c(n)>0$ depending only on $n$ such that for $t\leq S_1$,\\
\hspace*{0.5cm}\emph{(i)} $\varphi\leq (\log(1+c(n)KS_1)^n+1)t$\\
\hspace*{0.4cm}\emph{(ii)} $\dot{\varphi}\leq (\log(1+c(n)KS_1)^n+1)+n.$\\
\hspace*{0.3cm}\emph{(iii)}
\[\dot{\varphi}(x,t)\geq \frac{1}{S-S_1}\biggl[\underset{M}{\inf}\;u-\underset{M}{\sup}\;u-(\log(1+c(n)KS_1)^n+1+n)t\biggl],\]
\hspace*{1.1cm}and
\[\varphi(x,t)\geq \frac{1}{S-S_1}\int_0^t\biggl[\underset{M}{\inf}\;u-\underset{M}{\sup}\;u-(\log(1+c(n)KS_1)^n+1+n)s\biggl]ds.\]
where $u$ is the smooth bounded function defined in \emph{\textbf{(a2)}}.
\end{lemma}
\begin{proof}
As Lemma 3.4 of \cite{r4} (see \cite{r3} also),
for $\epsilon>0$, set $H=\varphi-At-\epsilon\rho$ where $A>0$ is a constant to be determined.
We find $H<0$ on $(M\setminus\Omega\times[0,S_1])$ for some compact set $\Omega$.
Therefore, provided $\sup_{M\times[0,S_1]}H>0$, there exists $(x_0,t_0)\in \Omega\times[0,S_1]$ with $t_0>0$ such that $H(x_0,t_0)=\sup_{M\times[0,S_1]}H$.
At $(x_0,t_0)$, we have $(\nabla d\varphi-\epsilon \nabla d\rho)\leq 0$, and
\[0\leq \frac{\partial}{\partial t}H=\log \frac{\det ({{g}_{0}}-t_0\beta ({{g}_{0}})+\nabla d\varphi )}{\det ({{g}_{0}})}-A.\]
Choose local coordinates around $x_0$ as in \textbf{(a3)}, we have
\[\begin{aligned}
-\beta (g_0)&=\frac{\partial^2 \log \det[g_0]}{\partial x^i\partial x^j}\\
&=\partial_j(g_0^{pq}\partial_i (g_0)_{pq})\\
&=g_0^{pq}\partial_j\partial_i (g_0)_{pq}-g_0^{pp'}(\partial_j (g_0)_{p'q'})g_0^{q'q}\partial_i (g_0)_{pq}\\
&=\partial_j\partial_i (g_0)_{pp}-(\partial_j (g_0)_{pq})\partial_i (g_0)_{pq}\\
&\leq c(n)Kg_0,
\end{aligned}\]
and then
\[\begin{aligned}
{{g}_{0}}-t_0\beta ({{g}_{0}})+\nabla d\varphi&\leq {{g}_{0}}-t_0\beta ({{g}_{0}})+\epsilon \nabla d\rho\\
&\leq (1+c(n)Kt_0+\epsilon C_1)g_0
\end{aligned}\]
for some positive constant $c(n)$ depending only on $n$ and a positive constant $C_1$ such that $\nabla d\rho\leq C_1g_0$. Here we have used the fact that $g_0$ satisfies \textbf{(a3)} and $|\nabla d\rho|_{g_0}$ is bounded.
Fixing $A=\log(1+c(n)KS_1)^n+1$, we get a contradiction when $\epsilon$ is small enough.
Therefore, $H\leq 0$ for any $\epsilon>0$. Letting $\epsilon\rightarrow 0$, we (i) follows.

(ii) By Lemmas \ref{l4.1}, \ref{l4.3} and the fact that $\varphi=0$ at $t=0$, we conclude that $t\dot{\varphi}-\varphi-nt\leq0$.
Thus, (ii) is proved by (i).

(iii) Let $\Lambda=(S -t)\dot{\varphi}+\varphi+nt$. By Lemma \ref{l4.1}, we have
\[\begin{aligned}
(\frac{\partial}{\partial t}-L_g)(\Lambda-u)&=-S {\rm tr}_g(\beta(g_0))+{\rm tr}_g g_0+L_g u\\
&\geq -{\rm tr}_g(g_0+\nabla du)+{\rm tr}_gg_0+L_g u\\
&=0.
\end{aligned}\]
Here we have used \textbf{(a2)}. Since $\Lambda$ and $u$ are bounded, we conclude by Lemma \ref{l4.3} that
\[\Lambda-u\geq \underset{M,t=0}{\inf}(\Lambda-u)=-\underset{M}{\sup}\;u.\]
Therefore, (iii) is valid by (i).
\end{proof}
\begin{lemma}\label{l4.5} Let $g_0$ be as in Lemma \ref{l4.4}. Then there are constants $c_1(n)$, $c_2(n)$ depending only on $n$ such that for $t\leq S_1<S_2<S$,
the solution $g(t)$ of \eqref{e1.2} on $M\times[0,S_1]$ which is uniformly equivalent to $g_0$, satisfies
\[{\rm tr}_{g_0}g\leq \exp\biggl[\log\biggl(\frac{1}{2}c_1K+\frac{1}{2}(c_1^2K^2+4c_2K^2A(1+2\mathfrak{m})^3)^{\frac{1}{2}}\biggl)+A\biggl]+n\]
on $M\times[0,S_1]$, where
\[A=\alpha^{-1}(2\mathfrak{m}+1)^2(c_1K+1) \]
and $\alpha=1-\frac{S_2}{S}$ and $\mathfrak{m}=\sup_{M\times[0,S_1]}|(S_2 -t)\dot{\varphi}+\varphi+nt-\frac{S_2}{S}u|$.
\end{lemma}
\begin{proof}
Similar to \cite{r4} and \cite{r3}, let $\Upsilon={\rm tr}_{g_0}g$ and $\Theta={\rm tr}_g g_0$.
We first estimate $(\frac{\partial}{\partial t}-L_g)\log \Upsilon$. In the proof of this lemma, we will denote a positive constant depending only on $n$ by $c_i$. By Lemma \ref{l4.2} we have
\[{\rm \uppercase\expandafter{\romannumeral1}}\leq \frac{1}{{\rm tr}_{g_0}g}\biggl(-2g_0^{li}g^{qk}(T_0)_{ik}^l\frac{\partial_q {\rm tr}_{g_0}g}{{\rm tr}_{g_0}g}+2g^{qk}\partial_k g_0^{pl}(T_0)_{lq}^p+g_0^{li}g^{jp}g^{qk}C_{ijk}C_{lpq}\biggl).\]
\[\begin{aligned}
{\rm \uppercase\expandafter{\romannumeral2}}=-\frac{1}{{\rm tr}_{g_0}g}\biggl(g^{ij}g_0^{kl}(\partial_l(T_0)_{ik}^j+\partial_i(T_0)_{jl}^k)+g^{ij}(\partial_i\partial_j g_0^{kl})g_{kl}\biggl).
\end{aligned}\]
In a local coordinates satisfying \textbf{(a3)}, we get
\[2g^{qk}\partial_k g_0^{pl}(T_0)_{lq}^p\leq \frac{1}{2}c_3 K \Theta,\]
\[-g^{ij}g_0^{kl}(\partial_l(T_0)_{ik}^j+\partial_i(T_0)_{jl}^k)\leq \frac{1}{2}c_3 K \Theta,\]
\[-g^{ij}(\partial_i\partial_j g_0^{kl})g_{kl}\leq \frac{1}{2}c_3 K\Upsilon\Theta,\]
and
\[\begin{aligned}
\frac{1}{{\rm tr}_{g_0}g}g_0^{li}g^{jp}g^{qk}C_{ijk}C_{lpq}\leq \frac{1}{2}c_3 K \Theta.
\end{aligned}\]
Hence
\begin{equation}
(\frac{\partial}{\partial t}-L_g)\log \Upsilon\leq 2\Upsilon^{-2}g^{qk}g_0^{li}(T_0)_{ki}^l\partial_q {\rm tr}_{g_0}g+c_3(K+\Upsilon^{-1}K)\Theta
\end{equation}
Fix $S_2<S$ and let $\Lambda=(S_2 -t)\dot{\varphi}+\varphi+nt$ and $u$ the function as in \textbf{(a2)}. Thus, by Lemma \ref{l4.1},
\[\begin{aligned}
(\frac{\partial}{\partial t}-L_g)(\Lambda-\frac{S_2}{S}u)&=-S_2 {\rm tr}_g(\beta(g_0))+\frac{S_2}{S}L_g u+\Theta\\
&\geq -\frac{S_2}{S}{\rm tr}_g(g_0+\nabla du)+\frac{S_2}{S}L_g u+\Theta\\
&=(1-\frac{S_2}{S})\Theta.
\end{aligned}\]
Let $\mathfrak{m}=\sup_M|\Lambda-\frac{S_2}{S}u|$ and $\phi=\Lambda-\frac{S_2}{S}u+\mathfrak{m}+1\geq 1$.
By the proof of Lemma \ref{l4.3}, there is a constant $C_1>0$ such that if $P=e^{2C_1t}\rho$, then
\[(\frac{\partial}{\partial t}-L_g)P\geq 0.\]
Finally, let
\[Q=\log \Upsilon+A\phi^{-1}-\epsilon P\]
where $A\geq 1$ is a constant to be determined later. We have
\begin{equation}\label{e4.11}
\begin{aligned}
(\frac{\partial}{\partial t}-L_g)Q&\leq 2\Upsilon^{-2}g^{qk}g_0^{li}(T_0)_{ki}^l\partial_q \Upsilon+c_3(K+\Upsilon^{-1}K)\Theta\\
&\;\;\;\;-A\phi^{-2}(\frac{\partial}{\partial t}-L_g)\phi-2A\phi^{-3}g^{ij}\partial_i\phi \partial_j\phi\\
&\leq 2\Upsilon^{-2}g^{qk}g_0^{li}(T_0)_{ki}^l\partial_q \Upsilon+c_3(K+\Upsilon^{-1}K)\Theta\\
&\;\;\;\;-A\alpha \phi^{-2}\Theta-2A\phi^{-3}g^{ij}\partial_i\phi \partial_j\phi
\end{aligned}
\end{equation}
where $\alpha=1-\frac{S_2}{S}>0$. Since $Q<0$ outside $\Omega\times[0,S_1]$ for a compact set $\Omega$ and $Q$ is bounded from above, there is a point $(x_0,t_0)$ with $x_0 \in \Omega$ such that $Q(x_0,t_0)=\sup_{M\times [0,S_1]}Q$. If $t_0=0$, we find
\begin{equation}
\sup_{M\times [0,S_1]}Q\leq A+\log n
\end{equation}
since $\phi\geq1$. For the case $t_0>0$, we have, at $(x_0,t_0)$,
\[\Upsilon^{-1}\partial_q \Upsilon=A\phi^{-2}\partial_q \phi+\epsilon\partial_qP.\]
We may assume $(g_0)_{ij} (x_0) =\delta_{ij}$ and $g_{ij} (x_0) =\lambda_i \delta_{ij}$ by a rotation if necessary. At $(x_0, t_0)$ we have
\[\begin{aligned}
2\Upsilon^{-2}g^{qk}g_0^{li}(T_0)_{ki}^l\partial_q \Upsilon&=2\Upsilon^{-1}g^{qk}g_0^{li}(T_0)_{ki}^l(A\phi^{-2}\partial_q \phi+\epsilon\partial_qP)\\
&\leq c_4K \Upsilon^{-1}\biggl(\underset{q}{\sum}\lambda_q^{-1}(A|\partial_q \phi|+\epsilon|\partial_qP|)\biggl)\\
&\leq c_4K \Upsilon^{-1}A\sqrt{(\underset{q}{\sum}\lambda_q^{-1})(\underset{q}{\sum}\lambda_q^{-1}|\partial_q \phi|^2)}+\epsilon C_2\Theta\\
&\leq 2A\phi^{-3}\underset{q}{\sum}\lambda_q^{-1}|\partial_q\phi|^2+c_5K^2A\Upsilon^{-2}\phi^3\underset{q}{\sum}\lambda_q^{-1}+\epsilon C_2\Theta\\
&=2A\phi^{-3}g^{ij}\partial_i\phi \partial_j\phi+c_5K^2A\Upsilon^{-2}\phi^3\Theta+\epsilon C_2\Theta
\end{aligned}\]
for some constant $C_2$ independent of $\epsilon$ and some constants $c_4,c_5>0$ depending only on $n$. By \eqref{e4.11}, we have, at $(x_0,t_0)$,
\[\begin{aligned}
0&\leq (\frac{\partial}{\partial t}-L_g)Q\\
&\leq c_3(K+\Upsilon^{-1}K)\Theta-A\alpha \phi^{-2}\Theta + c_5K^2A\Upsilon^{-2}\phi^3\Theta+\epsilon C_2\Theta.
\end{aligned}\]
Hence
\[0\leq c_5K^2A\Upsilon^{-2}\phi^3+c_3K\Upsilon^{-1}+(c_3K+\epsilon C_2-A\alpha \phi^{-2}).\]
Let $A=\alpha^{-1}(2\mathfrak{m}+1)^2(c_3K+\epsilon C_2+1)$, we find
\[0\leq c_5K^2A\Upsilon^{-2}\phi^3+c_3K\Upsilon^{-1}-1.\]
Since $\Upsilon^{-1}>0$ and $1\leq\phi\leq 1+2\mathfrak{m}$, we have
\[\Upsilon^{-1}\geq \frac{-c_3K+(c_3^2K^2+4c_5K^2A\phi^3)^{\frac{1}{2}}}{2c_5K^2A\phi^3}.\]
Therefore,
\[\Upsilon\leq \frac{1}{2}\biggl(c_3K+(c_3^2K^2+4c_5K^2A(1+2\mathfrak{m})^3)^{\frac{1}{2}}\biggl).\]
It follows that on $M\times[0,S_1]$,
\[Q\leq \log\biggl(\frac{1}{2}c_3K+\frac{1}{2}(c_3^2K^2+4c_5K^2A(1+2\mathfrak{m})^3)^{\frac{1}{2}}\biggl)+A.\]
As $\epsilon$ is arbitrary, the result follows.

\end{proof}
In the subsequent part of this section, we establish local derivative estimates for solutions of \eqref{e1.2} assuming local uniform bounds on the metric,
which can be seen as a real version of Chern-Ricci flow due to Sherman-Weinkove \cite{r5}. Since our estimates are local, we may work in a small open subset of $\mathbb{R}^n$.
Let $B_r$ denote the ball of radius $r$ centered at the origin in $\mathbb{R}^n$. Assume that $r$ is sufficiently small such that $\overline{B_r}$ is contained in a coordinate chart.
\begin{lemma}\label{l4.6} Fix $0<\varepsilon<2\varepsilon <T$, $0<r<\sqrt{\varepsilon}$,. Let $g(t)$ solve the geometric flow \eqref{e1.2} in a neighborhood of $B_r$ for $t\in [0,T)$. Assume $N>1$ satisfies
\[\frac{1}{N}g_0\leq g(t)\leq Ng_0\;\;\;\; on\; B_r\times[0,T).\]
Then there exist $C>0$, $0<\alpha<1$ depending only on $g_0,n,N,\varepsilon$ such that

\hspace*{0.2cm}\emph{(i)}
\begin{minipage}[t]{0.91\linewidth}
$[g_{ij}]_{\alpha/2,\alpha}\leq Cr^{-\alpha}$ on $B_{r/2}\times [\varepsilon,T)$.
\end{minipage}

\hspace*{0.1cm}\emph{(ii)}
\begin{minipage}[t]{0.91\linewidth}
$\|\varphi(t)\|_{C_k}\leq C_k$ on $B_{r/4}\times [2\varepsilon,T)$ for each $k\in\mathbb{N}$, where $C_k$ depends only on $g_0,n,N,\varepsilon,r$.
\end{minipage}
\end{lemma}
\begin{proof}Since the geometric flow is a concave fully nonlinear uniformly parabolic equation on $B_{r}\times [0,T)$,
the $C^{\alpha/2,\alpha}$ estimate for $g_{ij}$ follows from the Evans-Krylov theorem \cite{r8}(see also lemma 14.6 of \cite{r9}). We then obtain
\[\frac{{\rm osc}_{Q(r/2)}\varphi_t+{\rm osc}_{Q(r/2)}\nabla d\varphi}{(r/2)^\alpha}\leq C r^{-\alpha}({\rm osc}_{Q(r)}\varphi_t+{\rm osc}_{Q(r)}\nabla d)\]
where
\[Q(r)=B_r \times[0,T)\]
and
\[Q(r/2)=\{(x,t)\in B_r \times[0,T)|\;|x|<r/2,t_0-r^2/4\leq t\leq t_0\}\]
for any $t_0\in [\varepsilon,T)$ and (i) is proved.

To prove (ii), we work on the ball $B_{r/2}$. Consider the first order differential operator $D=\frac{\partial}{\partial x^\gamma}$ for $1\leq \gamma \leq n$.
From the equation \eqref{e1.3} we have
\begin{equation}\label{e4.13}
\frac{\partial}{\partial t}(D\varphi)=g^{ij}Dg_{ij}-g_0^{ij}D(g_0)_{ij}=g^{ij}\partial_i\partial_j(D\varphi)+g^{ij}D(\hat{g}_t)_{ij}-g_0^{ij}D(g_0)_{ij},
\end{equation}
where $\hat{g}_t:=g_0-t\beta(g_0)$. Furthermore, if we denote the largest and smallest eigenvalues of $g^{ij}$ by $\Lambda$ and $\lambda$ respectively, we have
\[C^{-1}\leq \lambda\leq\Lambda\leq C\]
for a uniform positive constant $C$.

Next,$g^{ij}$ is uniformly bounded in the $C^{\alpha/2,\alpha}$ parabolic norm for some $0<\alpha<1$ on $B_{r/2}\times [\varepsilon,T)$.
Since $g(t)$ is uniformly bounded it is easy to find that $u:=\frac{\partial \varphi}{\partial x^\gamma}$ in \eqref{e4.13} is bounded in the $C^0$ norm and $|\nabla d\varphi|_{C^0}$ is uniformly bounded. Moreover, $g^{ij}D(\hat{g}_t)_{ij}-g_0^{ij}D(g_0)_{ij}$ in \eqref{e4.13} is uniformly bounded in the $C^{\alpha/2,\alpha}$ norm.

We can then apply interior estimates (c.f. Theorem 8.11.1 in \cite{r2}) to \eqref{e4.13} to prove that $u$ is bounded in the parabolic $C^{1+\alpha/2,2+\alpha}$ norm on a slightly smaller parabolic domain: $[\varepsilon',T]\times B_{r'}$ for any $\varepsilon'$ and $r'$ such that $0<\varepsilon<\varepsilon'<2\varepsilon$ and $r/4<r'<r/2$. Tracing through the argument in \cite{r2}, one can check that the estimates we obtain indeed are of the desired form.

Applying $D$ to the equality $g_{ij}(t)=(\hat{g}_t)_{ij}+\partial_i\partial_j\varphi$, we get
\[Dg_{ij}=D(\hat{g}_t)_{ij}+\partial_i\partial_j u\]
where we recall that $D=\partial/\partial x^\gamma$ for $1\leq\gamma\leq n$.
Note that we have bounds for $u$ in $C^{1+\alpha/2,2+\alpha}$, which implies that $\partial_i\partial_j u$ is uniformly bounded in $C^{\alpha/2,\alpha}$.
Since $D(\hat{g}_t)_{ij}$ is uniformly bounded in all norms, we conclude that $Dg_{ij}$ is uniformly bounded in $C^{\alpha/2,\alpha}$ for all $i,j$.
It follows that $\partial_\gamma g^{ij}$ is uniformly bounded in $C^{\alpha/2,\alpha}$ for all $i,j,\gamma$.
We have a similar estimate for $\partial_\gamma f$. By Theorem 8.12.1 in \cite{r2} (with $k=1$), we see that for any $\delta$, $\partial_\delta u$ is uniformly
bounded in $C^{1+\alpha/2,2+\alpha}$ in a slightly smaller parabolic domain.
This means that $D^\delta \varphi$ is uniformly bounded in $C^{1+\alpha/2,2+\alpha}$ for any multi-index $\delta \in \mathbb{R}^n$ with $|\delta|\leq 2$.

We can then iterate this procedure and obtain the required $C^k$ bounded for $g(t)$ for all $k$. This completes the proof.

\end{proof}

\section{existence criteria for the geometric flow}

In this section, we consider the existence criteria for the geometric flow \eqref{e1.2}. Indeed, we shall prove Theorem \ref{t1.2}.
We note that the conditions (i) and (ii) in Theorem \ref{t1.2} are \textbf{(a1)} and \textbf{(a3)} in Page 7 respectively.
%
First it is easy to find $S_A\leq S_B$. Indeed, the flow \eqref{e1.2} can be rewritten as
\[\frac{\partial}{\partial t}g=-\beta (g_0)+\nabla d\alpha(t),\;\;\;\;\;\;{\rm with} \; \alpha(t)=\log\frac{\det g(t)}{\det g_0}.\]
Thus, as long as $g(t)$ solves the flow and is uniformly equivalent to $g_0$ in $M\times [0,S]$, the solution $g(t)$ must be of the form $g(t)=g_0-t\beta(g_0)+\nabla du$, for some smooth bounded function $u=u(x,t)$. It follows that
\[g_0-S\beta(g_0)+\nabla du=g(t)\geq \theta g_0\]
for some $\theta>0$ and $S_A\leq S_B$.

It suffices to prove that $S_B\leq S_A$. We consider $S>0$ such that \eqref{e5.1} holds for some bounded function $u$ and some $\theta>0$. 
We follow an idea of \cite{r4}.

Let $\kappa\in (0,1)$, $f:[0,1)\rightarrow [0,\infty)$ be the function:
\begin{equation}
\label{e5.2}
f(s)=\left\{ \begin{aligned}
&0,\;\;\;\;\;\;\;\;\;\;\;\;\;\;\;\;\;\;\;\;\;\;\;\;\;\;\;\;\;\;\;\;\;\;\;\;\;\;\;s\in[0,1-\kappa];\\
&-\log\biggl[1-\biggl(\frac{s-1+\kappa}{\kappa}\biggl)^2\biggl],s\in(1-\kappa,1).
\end{aligned} \right.
\end{equation}
Let $\psi\geq 0$ be a smooth function on $\mathbb{R}^+$ such that 
\begin{equation}
\label{e5.3}
\psi(s)=\left\{ \begin{aligned}
&0,\;\;\;s\in[0,1-\kappa+\kappa^2];\\
&1,\;\;\;s\in(1-\kappa+2\kappa^2,1).
\end{aligned} \right.
\end{equation}
and $\frac{2}{\kappa^2}\geq \psi'\geq 0$. Define
\[\mathfrak{F}:=\int_0^s\psi(\tau)f'(\tau)d\tau\]
By \cite{r6}, we have
\begin{lemma}Suppose $0<\kappa<\frac{1}{8}$. Then the function $\mathfrak{F}\geq 0$ defined above is smooth and satisfies the following:\\
\hspace*{0.5cm}\emph{(i)} $\mathfrak{F}=0$ for $s\in[0,1-\kappa+\kappa^2]$.\\
\hspace*{0.4cm}\emph{(ii)} $\mathfrak{F}'\geq 0$ and for any $k\geq 1$, $\exp(-k\mathfrak{F})\mathfrak{F}^{(k)}$ is uniformly bounded.\\
\hspace*{0.3cm}\emph{(iii)} For any $1-2\kappa<s<1$, there is $\tau>0$ with $0<s-\tau<s+\tau<1$ such that
\[1\leq \exp(\mathfrak{F}(s+\tau)-\mathfrak{F}(s-\tau))\leq 1+c_2\kappa;\;\;\;\tau\exp(\mathfrak{F}(s_0-\tau))\geq c_3\kappa^2\]
\hspace*{0.9cm} for some absolute constants $c_2>0$, $c_3>0$.
\end{lemma}

For any $\rho_0>0$, let $U_{\rho_0}$ be the component of $\{x|\rho(x)<\rho_0\}$ that contains the fixed point mentioned in {\textbf{(a1)}}. Hence $U_{\rho_0}$ will exhaust $M$ as $\rho_0\rightarrow \infty$.

For $\rho_0>>1$, we define $F(x)=\mathfrak{F}(\rho(x)/\rho_0)$. Let $h_0=e^{2F}g_0$. We find $(U_{\rho_0},h_0)$ is a complete Riemannian manifold
and $h_0=g_0$ on $\{\rho(x)<(1-\kappa+\kappa^2)\rho_0\}$ (see \cite{r7}).

Using same arguments of Lemma 4.3 in \cite{r4}, we have
\begin{lemma}\label{l5.3}$(U_{\rho_0},h_0)$ has bounded geometry of infinite order.
\end{lemma}

\begin{lemma}\label{l5.4} Suppose $g_0$ satisfies \emph{\textbf{(a3)}}. Then

\hspace*{0.2cm}\emph{(i)}
\begin{minipage}[t]{0.91\linewidth}
For any $\epsilon>0$ and any point $p \in M$, there exists a local coordinates $\{x^1,\ldots ,x^n\}$ around $p$ such that $(h_0)_{ij}(p)=\delta_{ij}$,
$|\partial_k(h_0)_{ij}(p)|\leq K+\epsilon$ $|\partial_i(\partial_j (h_0)_{kl}-\partial_l (h_0)_{kj})(p)|\leq K+\epsilon$ and $\partial_l \partial_k(h_0)_{ij}(p)\leq K+\epsilon$, provided $\rho_0$ is large enough. Furthermore, all changes of these local coordinate systems are affine.
\end{minipage}

\hspace*{0.1cm}\emph{(ii)}
\begin{minipage}[t]{0.93\linewidth}
For any $S>\epsilon>0$, there is $\rho_1>0$ such that if $\rho_0\geq \rho_1$, then
\[h_0-(S-\epsilon)\beta(h_0)+\frac{S-\epsilon}{S}\nabla du\geq \frac{S-\epsilon}{S}h_0\]
\end{minipage}

\emph{(iii)}
\begin{minipage}[t]{0.93\linewidth}
$(U_{\rho_0},h_0)$ satisfies \emph{\textbf{(a3)}} provided $\rho_0$ is large enough.
\end{minipage}

\end{lemma}
\begin{proof}To prove (i), for any point $p \in M^n$ be any point, let $\{y^1,\ldots ,y^n\}$ be the local coordinates around $p$ satisfying {\textbf{(a3)}} with respect to $g_0$. In these coordinates, $(h_0)_{ij}(y)=e^{2F}g_0(y)$. We can now define a new coordinate system $\{x^1,\ldots ,x^n\}$ around $p$ by the affine transformation $x^i=e^F(p)y^i$ for $i=1,\ldots ,n$. We have
\[\frac{\partial}{\partial x^i}=\frac{1}{e^F(p)}\frac{\partial}{\partial y^i},\;\;\;\;\;\;(h_0)_{ij}(x)=\frac{e^{2F}}{e^{2F}(p)}(g_0)_{ij}(y)\]
for any $i,j=1,\ldots ,n$, and
\begin{equation}
\begin{aligned}
\frac{\partial(h_0)_{ij}}{\partial x^k}\biggl|_p=\frac{1}{e^F}\frac{(g_0)_{ij}(y)}{\partial y^k}\biggl|_p
+2\rho_0^{-1}e^{-\mathfrak{F}}\mathfrak{F}'\frac{\partial \rho}{\partial y^k}(g_0)_{ij}(y)\biggl|_p
\end{aligned}
\end{equation}
for any $i,j,k=1,\ldots ,n$. Therefore, in the new coordinates $\{x^1,\ldots ,x^n\}$, we have $(h_0)_{ij}(p)=\delta_{ij}$, $|\partial_k(h_0)_{ij}(p)|\leq K+\epsilon$, provided $\rho_0$ is large enough.

Similarly, we find that $|\partial_i(\partial_j (h_0)_{kl}-\partial_l (h_0)_{kj})(p)|\leq K+\epsilon$, $\partial_l \partial_k(h_0)_{ij}(p)\leq K+\epsilon$ in the new coordinates $\{x^1,\ldots ,x^n\}$, provided $\rho_0$ is large enough.

To prove (ii), we may assume that $\theta\leq 1$. For any $\epsilon>0$, if $\rho_0$ is large enough,
we have
\[\begin{aligned}
h_0-(S-\epsilon)\beta(h_0)&=h_0-(S-\epsilon)\beta(g_0)+2n(S-\epsilon)\nabla dF\\
&\geq h_0+\frac{S-\epsilon}{S}(\beta-1)g_0-\frac{S-\epsilon}{S}\nabla du-\epsilon e^{2F}g_0\\
&\geq \biggl(1+\frac{S-\epsilon}{S}(\beta-1)-\epsilon \biggl)h_0-\frac{S-\epsilon}{S}\nabla du\\
&=\frac{S-\epsilon}{S}h_0-\frac{S-\epsilon}{S}\nabla du
\end{aligned}\]
since $h_0\geq g_0$.

By Lemma \ref{l5.3}, $(U_{\rho_0},h_0)$ has bounded geometry of infinite order and its Ricci curvature is bounded from below. By Lemma \ref{l3.3}, we obtain (iii).

\end{proof}
\begin{proof}[Proof of Theorem \ref{t1.2}]
Let $S>0$ be a constant such that \eqref{e5.1} is true.
Let $0<\epsilon<S$, $\rho_0$ large enough such that Lemma \ref{l5.4} is true and $S_\epsilon=S-\epsilon$.

\underline{{\textbf {Claim.}}} \eqref{e1.2} has solution $g_{\rho_0, \epsilon}(t)$ on $U_{\rho_0}\times [0,S_\epsilon)$, provided $\rho_0$ is large enough. Moreover, for any $\delta>0$ such that $0<S_\epsilon -\delta<S_\epsilon$, there is a constant independent of $\rho_0$ and $\epsilon$ such that
\[C^{-1}g_{\rho_0, \epsilon}(t)\leq g_0\leq Cg_{\rho_0, \epsilon}(t)\]
on $U_{\rho_0}\times [0,S_\epsilon-\delta)$.

If the claim is true, based on the estimates established in  Lemma \ref{l4.6}, as $\epsilon\rightarrow 0$, we can choose suitable
$\rho_0(\epsilon)\rightarrow \infty$ such that $g_{\rho_0 (\epsilon), \epsilon}(t)$
converges uniformly on any compact sets to a solution of the geometric flow \eqref{e1.2}, $g(t)$ on $M\times [0,S)$.
Moreover, for any $0<S'<S$, $g(t)$ is uniformly equivalent to $g_0$ on $M\times [0,S']$.
Thus, $S_B\leq S_A$ and Theorem \ref{t1.2} is proved.

Now we prove the claim. Since $h_0$ has bounded geometry of infinite order, the geometric flow \eqref{e1.2} admits a solution $h(t)$ on $U_{\rho_0}\times [0,S_1]$
for some $S_\epsilon >S_1>0$ and $h(t)$ is uniformly equivalent to $h_0$.
Choose $\delta>0$ such that $S_\epsilon-\delta>0$ and $S_1\leq S_\epsilon-\delta$.
Let $K$ be the bound as in Lemma \ref{l5.4}. By Lemma \ref{l4.4} and \ref{l4.5}, there is a constant $C_1$ depending only on $\delta,n,K,\mathfrak{m},S$ such that
\[\biggl|\log \frac{\det h(t)}{\det h_0}\biggl|\leq C_1, \;\;\;{\rm tr}_{h_0}h\leq C_1\]
on $U_{\rho_0}\times [0,S_1]$, where $\mathfrak{m} :=\sup_M |u|$.
Hence there is a constant $C_2$ depending only on $\delta,n,K,\mathfrak{m},S$ such that
\begin{equation}\label{e5.5}
C_2^{-1}h_0\leq h(t) \leq C_2h_0
\end{equation}
on $U_{\rho_0}\times [0,S_1]$. By the fact that $h_0$ has bounded geometry of infinite order again,
using the estimates in Lemma \ref{l4.6}, we conclude that all derivatives of $h(t)$ with respect to $h_0$ are uniformly bounded on $U_{\rho_0}\times [0,S_1]$.
Moreover, $h(S_1)$ also has bounded geometry of infinite order. By Lemma \ref{l3.5}, we find that $h(t)$ can be extended beyond $S_1$ to some $S_2$ with $S_1<S_2<S_\epsilon$ so that $h(t)$ is uniformly equivalent to $h_0$ on $U_{\rho_0}\times [0,S_2]$.
Combining this with \eqref{e5.5}, we conclude that the claim is true. This completes the proof of theorem \ref{t1.2}.

\end{proof}

Applying Theorem \ref{t1.2} to Hessian manifolds, we have
\begin{corollary}\label{c5.5}
Let $(M,g_0)$ be a complete noncompact Hessian manifold. If for any point $p \in M^n$, there exists a local coordinates $\{x^1,\ldots ,x^n\}$ around $p$ such that $(g_0)_{ij}(p)=\delta_{ij}$, $|\partial_k(g_0)_{ij}|(p)\leq K$, the Hessian curvature $|Q_{ijkl}(p)|\leq K$ for some constant $K$ independent of $p$ and all the changes
of these local coordinate systems are affine. Then $S_A=S_B$.
\end{corollary}
\begin{proof}By Proposition \ref{p2.4} and Lemma \ref{l3.3}, the Hessian manifolds in Corollary \ref{c5.5} satisfy {\textbf{(a1)}} and {\textbf{(a3)}}.
Thus, Corollary \ref{c5.5} follows.

\end{proof}

\section{Uniformization}

In this section, we study the asymptotic behavior of geometric flow on Hessian manifolds and prove that a complete noncompact Hessian manifold $(M^n,g)$ as in Corollary \ref{c5.5} with nonnegative Hessian bisectional curvature and its tangent bundle $(TM^n,g^T)$ has maximum volume growth is diffeomorphic to the Euclidean space $\mathbb{R}^n$.
This is obviously true for $\dim M=1$, because any one-dimensional manifold is diffeomorphic to either $\mathbb{R}$ or $S^1$ and $S^1$ is compact.

In higher dimensions, the problem seems more complicated. We shall explore the relationship between the flow \eqref{e1.2} and the Chern-Ricci flow.
Let $(M,\nabla, g_0)$ be a complete Hessian manifold of dimension $n$, and $g(t)$ is a solution of \eqref{e1.2}.
We find, by \eqref{e2.2}, Proposition \ref{p2.6} and \ref{p2.7},
\[
g^T (t) := \sum_{i,j=1}^{n}(g_{ij} (t)\circ \pi)dz^id\overline{z}^j
\]
(as in \eqref{e2.2}) is a solution to the K\"{a}hler-Ricci flow on the tangent bundle $(TM,J_\nabla,g_0^T)$
\begin{equation}\label{e6.2}
\left\{ \begin{aligned}
   &\frac{\partial }{\partial t}{g_{i\overline{j}}^T}(t)=-4R^T_{i\overline{j}} (g(t))  \\
      &g^T(0)= g_0^T.
\end{aligned} \right.
\end{equation}

\begin{lemma}\label{l6.1}
Suppose $(M,\nabla ,g)$ is a complete connected Hessian manifold. Then its tangent bundle $(TM,J_\nabla,g^T)$ is a complete K{\"a}hler manifold.
\end{lemma}
\begin{proof}
By Proposition \ref{p2.6}, we need only prove that $(TM,J_\nabla,g^T)$ is complete.
Let $\{P_j\} = \{(p_j, X_j)\} \subset TM$ be an arbitrary Cauchy sequence. For any $\varepsilon>0$, there exist a positive constant $N$ and
a family of curves $\gamma_{ij}:[a,b]\rightarrow TM$ with $\gamma_{ij}(a)=P_i,\gamma_{ij}(b)=P_j$ such that for all $i,j>N$, the length $L(\gamma_{ij})<\varepsilon$.
Now we consider the sequence $\{\pi (P_j)\} \subset M$ and the curves $\pi\circ\gamma_{ij}:[a,b]\rightarrow M$, where $\pi:TM\rightarrow M$ is the projection.
We find $\pi\circ\gamma_{ij}(a)=\pi(P_i),\pi\circ\gamma_{ij}(b)=\pi(P_j)$ and by straightforward calculations,
\begin{equation}
\begin{aligned}
d_g(\pi(P_i),\pi(P_j))\leq L(\pi \circ \gamma_{ij})&=\int_a^b|(\pi \circ \gamma_{ij})' (t)|_g dt\\
&\leq \int_a^b |(\pi \circ \gamma_{ij})' (t)|_{g^T}+|\gamma'_{ij}-(\pi \circ \gamma_{ij})' (t)|_{g^T} dt\\
&=L(\gamma_{ij})< \varepsilon
\end{aligned}
\end{equation}
provided $i,j > N$. It follows that $\{\pi (P_j)\}$ is a Cauchy sequence in $M$. By the completeness of $M$, there exists a point $p \in M$
such that $\pi(P_j)\rightarrow p$. Denote
\[
\ol{B(p,r)} := \{q \in M: d_g (q, p) < r\}.
\]
Fix a constant $r>0$ such that $\{\ol{B(p,r)},(x^1,\ldots,x^n)\}$ is a an affine coordinate system on $M$. Let $\lambda (q)$ and $\Lambda (q)$
be the minimal and maximal eigenvalues of $\{g_{ij} (q)\}$ for $q \in \ol{B(p,r)}$ respectively. There exists positive constants $\lambda_0$
and $\Lambda_0$ such that
\[
\lambda_0 \leq \lambda (q) \leq \Lambda (q) \leq \Lambda_0 \mbox{ for all } q \in \ol{B(p,r)}.
\]
Without loss of generality, we may assume $\{p_j\} \subset B(p,r)$ since $p_j\rightarrow p$.
Thus, $\{P_j\} \subset \pi^{-1} (\ol{B(p,r)})$. For any smooth curve $\gamma: [a, b] \rightarrow \pi^{-1} (\ol{B(p,r)})$, let
\[
\gamma' (t) = \gamma'_j (t) \frac{\partial}{\partial \xi^i} + \gamma'_{n+j} (t) \frac{\partial}{\partial \xi^{n+j}}.
\]
By the definition of $g^T$, we have,
\[
\sqrt{2\lambda_0}L_e (\gamma) \leq L (\gamma) 
= \int_a^b \sqrt{2g_{ij}\gamma'_i (t) \gamma'_j (t) + 2g_{ij}\gamma'_{n+i} (t) \gamma'_{n+j} (t)}| dt \leq \sqrt{2\Lambda_0} L_e (\gamma),
\]
where $L_e (\gamma)$ denotes the length of $\gamma$ with respect to the Euclidean metric on $\pi^{-1} (\ol{B(p,r)})$.
Denote $d_e$ to be the Euclidean distance on $\pi^{-1} (\ol{B(p,r)})$. It follows that for any $P, Q \in \pi^{-1} (\ol{B(p,r)})$,
\[
\sqrt{2\lambda_0} d_e (P,Q) \leq d_{g^T} (P, Q) \sqrt{2\Lambda_0} \leq d_e (P, Q)
\]
which means that the distances $d_e$ and $d_{g^T}$ are equivalent. Therefore, $(\pi^{-1} (\ol{B(p,r)}), d_{g^T})$ is complete.
It follows that
there exists a point $P = (p, X) \in \pi^{-1} (\ol{B(p,r)})$ such that $P_j \rightarrow P$ and $(TM,J_\nabla,g^T)$ is complete.
\end{proof}

\begin{theorem}\label{t6.2}
Let $(M, g_0)$ be a complete noncompact Hessian manifold satisfying the conditions in Corollary \ref{c5.5} with nonnegative Hessian bisectional curvature.
Then the geometric flow \eqref{e1.2} has a long time solution $g (t)$ on $M\times [0,\infty)$, and
for any $t\geq 0$, $g(t)$ is a Hessian metric with nonnegative Hessian bisectional curvature.
\end{theorem}
\begin{proof}
First we note that if $g(t)$ is a solution of \eqref{e1.2}, then $g^T(t)$ is a solution of \eqref{e6.2}.
By Proposition \ref{p2.7}, $g_0^T$ is K{\"a}hler with nonnegative holomorphic bisectional curvature. Furthermore, using Lemma \ref{l6.1} and Shi's result \cite{shi97}, we have $g^T(t)$ is K{\"a}hler with nonnegative holomorphic bisectional curvature, hence $g(t)$ is Hessian with
nonnegative Hessian bisectional curvature.

Now we consider the solvability of \eqref{e1.2}. Since $\beta(g_0)\leq 0$, by Corollary \ref{c5.5}, we have $S_A=S_B=\infty$.
This implies the existence of a family solutions $g_n(t)$ defined on $M\times [n,n+1]$ of \eqref{e1.2} satisfying $g_n(n)=g_{n-1}(n)$
and $g_0(0)=g_0$ for $n=0,1,2,\ldots$. Consequently, there exists a long time solution $g(t)$ such that
$g(t)|_{[n,n+1]}=g_n(t)$.

\end{proof}

Now we consider the following normalization of \eqref{e1.2}
\begin{equation}\label{e6.7}
\frac{\partial }{\partial t}{g_{ij}}(t)=-\beta_{ij} (g(t))-g_{ij}(t).
\end{equation}
It is easy to verify that if $\widetilde{g}(t)$ solves \eqref{e1.2}, then
\begin{equation}\label{e6.8}
g(t)=e^{-t}\widetilde{g}\left(e^t\right)
\end{equation}
is a solution to \eqref{e6.7}. We note that $g(t)$ in \eqref{e6.8} is defined for $-\infty<t<\infty$
and $g^T (t)$ satisfies the flow
\begin{equation}
\label{e6.13}
\frac{\partial }{\partial t}{g_{i\overline{j}}^T}(t)=- 4R^T_{i\overline{j}} (g(t)) - {g_{i\overline{j}}^T}(t)
\end{equation}
on $TM$.

\begin{proposition}\label{p6.3}
Let $(M, g_0)$ be as in Theorem \ref{t6.2} such that $(TM,J_\nabla,g_0^T)$ has maximum volume growth, i.e.,  $(TM,J_\nabla,g_0^T)$
satisfies \eqref{MVG}. Let $g(t)$ be the solution to \eqref{e6.7} which is defined in \eqref{e6.8}.
Then there exist positive constants $r_1,r_2$ and $C_2$ depending only on the initial metric such that for each $p\in M,t\in [0,\infty)$, there exists a smooth
map $\Phi_t$ from the Euclidean ball $\widehat{B}_0(r_1)$ centered at the origin of $\mathbb{R}^n$ to $M$ satisfying

\hspace*{0.3cm}\emph{(i)} $\Phi_t$ is a diffeomorphism from $\widehat{B}_0(r_1)$ onto its image;

\hspace*{0.2cm}\emph{(ii)} $\Phi_t(0)=p$;

\hspace*{0.1cm}\emph{(iii)} $\Phi^*_t(g(t))(0)=g_\epsilon$;

\hspace*{0.1cm}\emph{(iv)} $\frac{1}{r_2}g_\epsilon\leq \Phi^*_t(g(t))\leq r_2g_\epsilon$ in $\widehat{B}_0(r_1)$;

\noindent where $g_\epsilon$ is the standard metric on $\mathbb{R}^n$. Moreover, the following are true:

\hspace*{0.2cm}\emph{(v)}
\begin{minipage}[t]{0.9\linewidth}
For any $t_k \rightarrow \infty$ and for any $0<r<r_1$, the family $\{\Phi_{t_k}(\widehat{B}_0(r))\}_{k\geq 1}$ exhausts $M$ and hence $M$ is simply connected.
\end{minipage}

\hspace*{0.1cm}\emph{(vi)}
\begin{minipage}[t]{0.9\linewidth}
If $T$ is large enough, then $F_{i+1}=\Phi^{-1}_{(i+1)T}\circ \Phi_{iT}$ maps $\widehat{B}_0(r_1)$ into $\widehat{B}_0(r_1)$ for each $i$, and there is $0<\delta<1$, $0<a<b<1$ such that
\[|F_{i+1}(z)|\leq \delta|x|\]
for all $x\in \widehat{B}_0(r_1)$, and
\[a|v|\leq |F'_{i+1}(0)(v)|\leq b|v|\]
for all $v$ and all $i$.
\end{minipage}

\end{proposition}
\begin{proof}
Since $g^T (t)$ satisfies the equation \eqref{e6.13}, by Corollary 2.1 of \cite{Tam06}, there exists a constant $r_0>0$ depending only the
initial metric $g^T_0$ such that
the injectivity radius of $(TM,J_\nabla,g^T(t))$ is bounded below by $r_0$. For each $t \geq 0$ and $p \in M$, let
\[
\exp_P^T: \widehat{B}_0^T(r_0) \subset T_P(TM) \rightarrow B^T_P (r_0)
\]
be the exponential map on $(TM, g^T (t))$ near $P = (p, 0) \in TM$,
where $B^T_P (r_0)$ is a geodesic ball centered at $P$ with radius $r_0$ in $(TM, g^T(t))$.
We choose a complex normal coordinates $\{x_1, \ldots, x_{2n}\}$ on $B^T_P (r_0)$, such that $g^T_{i\ol j} (t,P) = \delta_{ij}$, $dg^T_{i\ol j} (t,P) = 0$
and $\partial z_k \partial z_l g^T_{i\ol j} (t,P) = 0$. By \cite{Ha95}, the components of the metric $g^T(t)$ with respect to coordinates $x_i$ satisfy
\[|\delta_{ij}-g^T_{ij}(t)|\leq C_2|x|^2,\;\;\;\frac{1}{2}\delta_{ij}\leq g^T_{ij}(t)\leq 2\delta_{ij}\]
in $B_P^T(r_1)$ for some positive constant $r_1,C_2$ depending only on $C_1, r_0$ and $n$, where $|x|^2=\sum_i(x_i)^2$.
Note that
\[
\begin{aligned}
F: M & \rightarrow TM\\
p & \mapsto (p, 0)
\end{aligned}
\]
is an embedding. In this setting we claim that $\exp_P^T|_{T_p M}=\exp_p$, where $\exp_p$ is the exponential map on $(M, g(t))$ near $p$.
Indeed, it is easy to find that $(M, g(t))$ is a totally geodesic submanifold of $(TM, g^T(t))$. It follows that any geodesic in $(M, g(t))$
is also a geodesic in $(TM, g^T(t))$.
Therefore, we have $\exp_P^T|_{T_pM}=\exp_p$.
Furthermore, we have $B_p (r_0) \subset B_P^T (r_0)$ and $(\exp_P^T)^{-1} (B_p (r_0)) \subset T_p M$ by the completeness of $M$,
where $B_p (r_0)$ is the geodesic ball centered at $p$ with radius $r_0$ in $(M, g(t))$.
Thus, the map $\Phi_t:=\exp_p|_{\widehat{B}_0(r_1)}$ satisfies conditions (i)-(iv).

Using same arguments to Corollary 2.2 in \cite{Tam06}, we can prove (v) and (vi).

\end{proof}

In the following, we will study the asymptotic behavior of the geometric flow and use the maps $\Phi_t$ to construct a diffeomorphism from $M$ to $\mathbb{R}^n$.

\begin{proposition}\label{p6.4}
Let $(M, g_0)$, and $g(t)$ be as in Proposition \ref{p6.3}. Let $p\in M$ be a fixed point in $M$ and $\lambda_1(t)\geq \cdot\cdot\cdot \geq \lambda_n(t)\geq 0$
be the eigenvalues of $\beta_{ij}(p,t)$ with respect to $g_{ij}(p,t)$. Then we have

\hspace*{0.1cm}\emph{(i)} For $1\leq i\leq n$ the limit $\lim_{t\rightarrow \infty} \lambda_i(t)$ exists.

\emph{(ii)}
\begin{minipage}[t]{0.92\linewidth}
Let $\mu_1>\cdot\cdot\cdot >\mu_l\geq 0$ be the distinct limits in (i) and $\rho>0$ be a constant
such that $[\mu_k-\rho,\mu_k+\rho]$, $1\leq k\leq l$ are disjoint. For any $t$, let $E_k(t)$ be the sum of the eigenspaces corresponding to the
eigenvalues $\lambda_i(t)$ such that $\lambda_i(t) \in (\mu_k-\rho,\mu_k+\rho)$. Let $P_k(t)$ be the orthogonal projection (with respect to $g(t)$)
onto $E_k(t)$. Then there exists $T>0$ such that if $t>T$ and if $w\in T_p(M)$, $|P_k(t)(w)|_t$ is continuous in $t$,
where $|\cdot|_t$ is the length measured with respect to the metric $g(p,t)$.
\end{minipage}
\end{proposition}
\begin{proof}
By Proposition 3.1 in \cite{Tam06} and the relationship between the flows \eqref{e6.7} and \eqref{e6.13}, the proposition follows.

\end{proof}

For any nonzero vector $v\in T_p(M)$, let $v(t)=v/|v|_t$ where $|v|_t$ is the length of $v$ with respect to $g(t)$ and let $v_i(t)=P_i(t)v(t)$.
Similar to Theorem 4.1 in \cite{Tam06}, we have

\begin{theorem}\label{t6.5}
Let $g(t)$ be as in Proposition \ref{p6.3}. With the same notation as above, $V=T_p(M)$ can be decomposed orthogonally with respect to $g(0)$ as $V_1\oplus \cdot\cdot\cdot \oplus V_l$ so that the following are true:

\hspace*{0.1cm}\emph{(i)}
\begin{minipage}[t]{0.92\linewidth}
If $v$ is a nonzero vector in $V_i$ for some $1\leq i\leq l$, then $\lim_{t\rightarrow \infty} |v_i(t)|=1$ and thus $\lim_{t\rightarrow \infty} \beta(v(t),v(t))=\mu_i$ and
\[\underset{t\rightarrow \infty}{\lim}\frac{1}{t}\log\frac{|v|^2_t}{|v|^2_0}=-\mu_i-1\]
Moreover, the convergences are uniform over all $v\in V_i\setminus\{0\}$.
\end{minipage}

\emph{(ii)}
\begin{minipage}[t]{0.92\linewidth}
For $1\leq i,j\leq l$ and for nonzero vectors $v\in V_i$ and $w\in V_j$ where $i\neq j$, $\lim_{t\rightarrow \infty}\langle v(t),w(t)\rangle_t=0$ and the convergence is uniform over all such nonzero vectors $v,w$.
\end{minipage}
\end{theorem}
\begin{proof}
Since we can use the same procedure of Theorem 4.1 in \cite{Tam06} with only minor modifications to prove Theorem \ref{t6.5},
we omit the proof here. The reader is referred to \cite{Tam06} for details.

\end{proof}

Given $T>0$, we define $F_{i+1}=\Phi^{-1}_{(i+1)T}\circ \Phi_{iT}$.
Then for each $i$, $F_i$ is a smooth map from $D(r_1)$ into $\mathbb{R}^n$ and is a diffeomorphism onto its image,
where $D(r_1) :=\{x\in \mathbb{R}^n: |x|<r_1\}$.
Let $A_i=F'_i(0)$ be the Jacobian matrix of $F_i$ at 0. By Proposition \ref{p6.3}, $T>0$ can be chosen large enough such that
\begin{equation}\label{e6.11}
F_i(D(r_1))\subset D(r_1),\;|F_i(x)|\leq \delta|x|\;\;{\rm for\; some}\;\;0<\delta<1.
\end{equation}
Since $\beta_{ij}\geq 0$ for all $t$, we have
\begin{equation}\label{e6.12}
a|v|\leq |A_i(v)|\leq b|v|
\end{equation}
for some $0<a<b<1$ and for all $i$. Note that $a,b,\delta$ are independent of $i$.
Now we construct a global diffeomorphism from $M$ to $\mathbb{R}^n$ as in the complex case \cite{Tam06}.

First, we reverse the order of $\lambda_i$ and hence the order of $\mu_k$ in Proposition \ref{p6.4} for convenience, where
$0\leq \lambda_1(t)\leq \lambda_2(t)\leq \cdot\cdot\cdot\leq \lambda_n(t)$ denote
the eigenvalues of $\beta_{ij}(t)$ with respect to $g(t)$ and $0\leq \mu_1<\mu_2 \cdot\cdot\cdot<\mu_l$ are their limits.
Let $\rho>0$ and $E_k(t)$, $1\leq k\leq l$ be as in Proposition \ref{p6.4} and $P_k(t)$
be the orthogonal projection onto the space $E_k(t)$ with respect to $g(t)$. Let $\tau_k=e^{-(\mu_k+1)T}$, $1\leq k\leq l$.

By Theorem \ref{t6.5}, $T_p(M)$ can be decomposed orthogonally with respect to the initial metric as $E_1\oplus \cdot\cdot\cdot \oplus E_l$ such that if $v\in E_k$ and $w\in E_j$ are nonzero vectors and $v(t)=v/|v|_t$, $w(t)=w/|w|_t$ where $|\cdot|_t$ is the norm taken with respect to $g(t)$, then
\begin{equation}
\underset{t\rightarrow \infty}{\lim}|P_k(t)v(t)|_t=1,\;\;{\rm for}\;\;1\leq k\leq l\;\;{\rm and}\;\;\underset{t\rightarrow \infty}{\lim}\langle v(t),w(t)\rangle_t=0\;\;{\rm for \;all}\;\;j\neq k,
\end{equation}
where $\langle \cdot,\cdot\rangle_t$ is the inner produce with respect to $g(t)$. Furthermore, the convergence is uniform.

For any $i$, define $E_{i,k}=d\Phi_{iT}^{-1}(E_k)$, $1\leq k\leq l$, $A(i)=A_i\cdot\cdot\cdot A_1$ and $A(i+j,i)=A_{i+j}\cdot\cdot\cdot A_{i+1}$.
Then we have $E_{i,k}=A(i)(E_{1,k})$ and $A_{i+1}(E_{i,k})=E_{i+1,k}$.

The following Lemma \ref{l6.6}-\ref{l6.10} can be proved using almost the same methods as Lemma 5.1-5.5 respectively in \cite{Tam06} where
the complex case was considered, so we omit their proofs.
\begin{lemma}
\label{l6.6}
Given $\epsilon>0$, there exists $i_0$ such that if $i\geq i_0$, then the following are true:

\hspace*{0.1cm}\emph{(i)}
\begin{minipage}[t]{0.92\linewidth}
$(1-\epsilon)\tau_k|v|^2\leq |A_{i+1}(v)|^2\leq (1+\epsilon)\tau_k|v|^2$ for all $v\in E_{i,k}$ and $1\leq k\leq l$, where $\tau_k=e^{-(\mu_k+1)T}$.
\end{minipage}

\emph{(ii)}
\begin{minipage}[t]{0.92\linewidth}
For any nonzero vector $v\in \mathbb{R}^n$
\[(1-\epsilon)\leq \frac{|v|^2}{\sum_{k=1}^l|v_k|^2}\leq (1+\epsilon)\]
where $v=\sum_{k=1}^l v_k$ is the decomposition of $v$ in $E_{i,1}\oplus \cdot\cdot\cdot \oplus E_{i,l}$.
\end{minipage}
\end{lemma}

We use similar notations as \cite{Tam06}. Let $\Phi$ be a polynomial map from $\mathbb{R}^n$ into $\mathbb{R}^n$,
which means that each component of $\Phi$ is a polynomial function.
Suppose $\Phi$ is of homogeneous of degree $m$, meaning that each component of $\Phi$ is a
homogeneous polynomial of degree $m\geq 1$. We define
\[\|\Phi\|=\underset{v\in \mathbb{R}^n,v\neq 0}{\sup}\frac{|\Phi(v)|}{|v|^m}.\]
If, in general, $\Phi$ is a polynomial map with $\Phi(0)=0$, let $\Phi=\sum_{m=1}^q \Phi_m$ be the decomposition of $\Phi$
such that $\Phi_m$ is homogeneous of degree $m$, $\|\Phi\|$ is defined by
\[\|\Phi\|=\sum_{m=1}^q \|\Phi_m\|.\]
If we decompose $\mathbb{R}^n$ as $E_{i,1}\oplus \cdot\cdot\cdot \oplus E_{i,l}$, we will denote $\mathbb{R}^n$ by $\mathbb{R}^n_i$. Let $\Phi:\mathbb{R}^n_i\rightarrow \mathbb{R}^n_{i+1}$ be a map. We can decompose $\Phi$ as $\Phi(v)=\sum_{k=1}^l \Phi_k(v)=\Phi_1\oplus \cdot\cdot\cdot\oplus \Phi_l$ where $\Phi_k(v)\in E_{i+1,k}$. Let $\alpha=(\alpha_1,\cdot\cdot\cdot,\alpha_l)$ be a multi-index such that $|\alpha|=\sum_{k=1}^l \alpha_k=m\geq 1$.
A polynomial map $\Phi$ is said to be \emph{homogeneous of degree} $\alpha$ if
\[\Phi(c_1v_1\oplus\cdot\cdot\cdot\oplus c_lv_l)=c^{\alpha}\Phi(v_1\oplus\cdot\cdot\cdot\oplus v_l),\]
where $v_k \in E_{i,k}$. Thus, if $\Phi$  homogeneous of degree $\alpha$,
then $\Phi$ is homogeneous of degree $|\alpha|$ in the usual sense.
$\Phi$ is said to be \emph{lower triangular}, if $\Phi_k(v_1\oplus\cdot\cdot\cdot\oplus v_l)=c_kv_k+\Psi_k(v_1\oplus\cdot\cdot\cdot\oplus v_{k-1})$.

\begin{lemma}
Let $\Phi:\mathbb{R}^n_i\rightarrow \mathbb{R}^n$ be homogeneous of degree $\alpha=(\alpha_1,\cdot\cdot\cdot,\alpha_l)$ with $|\alpha|=m$. Then
\[|\Phi(v_1\oplus\cdot\cdot\cdot\oplus v_l)|\leq l^m\|\Phi\||v_1|^{\alpha_1}\cdot\cdot\cdot|v_l|^{\alpha_l}.\]
Here by convention if $\alpha_i=0$, then $|v_i|^{\alpha_i}=1$ for all $v_i$.
\end{lemma}

Note that $\tau_1>\cdot\cdot\cdot>\tau_l$. Choose $1>\epsilon>0$ small enough such that $b^2(1-\epsilon)^{-1}(1+\epsilon)<1$ where $b<1$ is the constant in \eqref{e6.12}. Since we are interested in the maps $F_i$ for large $i$, without loss of generality, we assume the conclusions of Lemma \ref{l6.6} are true for all $i$ with this $\epsilon$. Let $m_0\geq 2$ be a positive integer such that $a^{-1}b^{m_0}<\frac{1}{2}$, where $0<a<b<1$ are the constants in \eqref{e6.12}.
The following lemmas aim to assemble the maps $F_i$ to produce a global diffeomorphism from $M$ to $\mathbb{R}^n$.

\begin{lemma}\label{l6.8}
Let $\Phi_{i+1}:\mathbb{R}^n_i\rightarrow \mathbb{R}^n_{i+1}$, $1\leq i<\infty$, be a family homogeneous polynomial maps of degree $m\geq 2$ such that $\sup_i\|\Phi\|<\infty$. Then there exist homogeneous polynomial maps $H_{i+1}$ and $Q_{i+1}$ from $\mathbb{R}^n_i$ to $\mathbb{R}^n_{i+1}$ such that $\Phi_{i+1}=Q_{i+1}+H_{i+1}-A_{i+2}^{-1}H_{i+2}A_{i+1}$. Moreover, $H_{i+1}$ and $Q_{i+1}$ satisfy the following:

\hspace*{0.2cm}\emph{(i)} $\sup_i\|H_i\|<\infty$ and $\sup_i\|Q_i\|<\infty$.

\hspace*{0.1cm}\emph{(ii)} $Q_{i+1}=0$ if $m\geq m_0$.

\emph{(iii)}
\begin{minipage}[t]{0.92\linewidth}
$Q_{i+1}$ is lower triangular:
\[Q_{i+1}(v_1\oplus\cdot\cdot\cdot\oplus v_l)=0\oplus Q_{i+1,2}(v_1)\oplus Q_{i+1,3}(v_1\oplus v_2)\oplus\cdot\cdot\cdot\oplus Q_{i+1,l}(v_1\oplus\cdot\cdot\cdot\oplus v_{l-1})\]
where $v_k\in E_{i,k}$ and $Q_{i+1,k}:\mathbb{R}^n_i\rightarrow E_{i+1,k}$.
\end{minipage}

\end{lemma}
\begin{lemma}\label{l6.9}
Given any $m\geq 2$, we can find constants $C(m)>0$ and $r_1\geq r_m>0$ and families of smooth maps $T_{i,m}$ from $D(r_m)\subset \mathbb{R}^n_i$ to $D(r_m)\subset \mathbb{R}^n_i$ and $G_{i+1,m}$ from $\mathbb{R}^n_i$ to $\mathbb{R}^n_{i+1}$ with the following properties:

\hspace*{0.2cm}\emph{(i)}
\begin{minipage}[t]{0.92\linewidth}
For each $i$, $T_{i+1,m}$ is a polynomial map of degree $m-1$ which is diffeomorphic to its image, $T_{i+1,m}(0)=0$, $T'_{i+1,m}=Id$
and $\|T_{i+1,m}\|\leq C(m)$.
\end{minipage}

\hspace*{0.1cm}\emph{(ii)}
\begin{minipage}[t]{0.92\linewidth}
$G_{i+1,n}=A_{i+1}+\widetilde{G}_{i+1,m}$ where $\widetilde{G}_{i+1,m}$ is a polynomial map of degree $m-1$.
\[\widetilde{G}_{i+1,m}(v_1\oplus\cdot\cdot\cdot\oplus v_l)=0\oplus \widetilde{G}_{i+1,m,2}(v_1)\oplus\cdot\cdot\cdot\oplus \widetilde{G}_{i+1,m,l}(v_1\oplus\cdot\cdot\cdot\oplus v_{l-1})\]
is lower triangular, and $\|G_{i+1,m}\|\leq C(m)$, $\widetilde{G}_{i+1,m}=0$ and $\widetilde{G}'_{i+1,m}=A_{i+1}$. Moreover, $G_{i+1,m}=G_{i+1,m_0}$
for all $m\geq m_0$, where $m_0$ is the integer in Lemma \ref{l6.8}.
\end{minipage}

\emph{(iii)} $F_{i+1}(D(r_m))\subset D(r_m)$ and
\[|T_{i+1,m}F_{i+1}(v)-G_{i+1,m}T_{i,m}(v)|\leq C(m)|v|^m.\]
\end{lemma}

Let $m\geq m_0$ and denote $G_{i+1,m}$ simply by $G_{i+1}$ and $\widetilde{G}_{i+1,m}$ by $\widetilde{G}_{i+1}$ etc.
Note that $G_{i+1}$ a diffeomorphism on $\mathbb{R}^n$ independent of $m$ and
degree of each $G_{i+1}$ is $m-1$. For any positive integers $i,j$, let $G(i+j,i)=G_{i+j}\cdots G_{i+1}$.
\begin{lemma}
\label{l6.10}
Let $G_{i+1}$ be as above, then its inverse is a polynomial map of degree $(m-1)^{l-1}$ and satisfies:
\[G^{-1}_{i+1}=A^{-1}_{i+1}+S_{i+1}\]
where $S_{i+1}:\mathbb{R}^n_{i+1}\rightarrow \mathbb{R}^n_{i}$ is defined by
\[S_{i+1}(w_1\oplus\cdot\cdot\cdot\oplus w_l)=0\oplus S_{i+1,2}(w_1)\oplus\cdot\cdot\cdot\oplus S_{i+1,l}(w_1\oplus\cdot\cdot\cdot\oplus w_{l-1}).\]
Moreover, $\|G^{-1}_{i+1}\|$ is bounded by a constant independent of $i$.
\end{lemma}
\begin{lemma}\label{l6.11}
Let $D(R)$ be a ball in $\mathbb{R}^n$ centered at the origin with radius $R$. Then we have

\hspace*{0.2cm}\emph{(i)}
\begin{minipage}[t]{0.92\linewidth}
There exists $\beta>0$ such that for all $x,x'\in D(R)$ and any positive integers $i$ and $j$,
\[|G(i+j,i)^{-1}(x)-G(i+j,i)^{-1}(x')|\leq \beta^j|x-x'|.\]
\end{minipage}

\hspace*{0.1cm}\emph{(ii)} For any positive integer $i$ and any open set $U$ containing the origin,
\[\bigcup_{j+1}^\infty G(i+j,i)^{-1}(U)=\mathbb{R}^n.\]
\end{lemma}
\begin{proof}
Let $f:\mathbb{R}^n\rightarrow \mathbb{R}^n$ be a  polynomial map, then there exists a polynomial map $\widehat{f}:\mathbb{C}^n\rightarrow \mathbb{C}^n$
such that $\widehat{f}(Re(z^1(p),\ldots,Re(z^n(p))=f(Re(z^1(p),\ldots,Re(z^n(p))$ for any $p\in \mathbb{C}^n$, where $\{z^1,\ldots,z^n\}$ is a
complex coordinate system of $\mathbb{C}^n$.
For a subspace $E$ of $\mathbb{R}^n$, we define $\widehat{E} := E \otimes_{\mathbb{R}} \mathbb{C}
= \mathrm{span}_{\mathbb{C}}\{v_1,\ldots,v_k\} \subset \mathbb{C}^n$, where $\{v_1,\ldots,v_k\}$ is a basis of $E$.

(i) For simplicity, we assume that $R=1$ and $i=0$. For general $i$, the proof is similar. Note that the constants
in the proof do not depend on $i$. Write
\begin{equation}
\widehat{G(j,0)}^{-1}=\widehat{G_1}^{-1}\cdots \widehat{G_j}^{-1}=H_{j,1}\oplus\cdots\oplus H_{j,l}
\end{equation}
with $H_{j,k}(v)\in \widehat{E_{1,k}}$. By Lemma \ref{l6.10}, we have
\[\widehat{G^{-1}_{i+1}}=\widehat{A^{-1}_{i+1}}+\widehat{S_{i+1}},\]
where $\widehat{S_{i+1}}:\widehat{\mathbb{R}^n_{i+1}}\rightarrow \widehat{\mathbb{R}^n_{i}}$ is defined by
\[\widehat{S_{i+1}}(w_1\oplus\cdots\oplus w_l)=0\oplus \widehat{S_{i+1,2}}(w_1)\oplus\cdots\oplus \widehat{S_{i+1,l}}(w_1\oplus\cdots\oplus w_{l-1}).\]
Here $w_1\oplus\cdots\oplus w_l=\widehat{E_{i+1,1}}\oplus\cdots\oplus \widehat{E_{i+1,l}}=\widehat{\mathbb{R}^n_{i+1}}$.

\underline{{\textbf {Claim.}}} If $f:\mathbb{R}^n\rightarrow \mathbb{R}^n$ is a s homogeneous polynomial map of degree $m$, then there exists a constant $C$ depending only on $\|f\|$ and $m$ such that $\|\widehat{f}\|\leq C$.

If the claim is true, by Lemma \ref{l6.10}, we can conclude that $\widehat{G_j^{-1}}$, $\widehat{A_j^{-1}}$ and $\widehat{S_{j}}$ are bounded by a constant independent of $j$. Then we have
\[|H_{j,k}(v)|\leq \beta^j\]
for some constant $\beta$ for all $k$ and $j$ provided $|v|\leq 1$. The proof for this is similar to that of Lemma 5.6 in \cite{Tam06}.
By applying the Schwartz lemma, (i) is proved.

For simplicity, we denote any constant depending on $\|f\|$ and $m$ by $C$. To prove the claim we only need to prove that any component $\widehat{f_k}$ of $\widehat{f}$ satisfies
$\|\widehat{f_k}\|\leq C$. Let $\{v_1,\ldots,v_n\}$ be a basic of $\mathbb{R}^n$, and
$\alpha=(\alpha_1,\cdots,\alpha_n)$ a multi-index such that $|\alpha|=\sum_{k=1}^n \alpha_k=m\geq 1$.
We define $E_1=\mathrm{span}_{\mathbb{R}}\{v_1\}$ and $E'_1=\mathrm{span}_{\mathbb{R}}\{v_2,\cdots,v_n\}$. If $w_1\in E_1$, $w_2\in E'_1$, and $|w_1|,|w_2|\leq 1$, we have
\begin{equation}\label{e6.15}
f_k(c_1w_1\oplus c_2w_2)=\sum_{i=0}^m c_1^i c_2^{m-i} f_{k,(i,m-i)}(w_1\oplus w_2)\leq C
\end{equation}
for any $-1\leq c_1,c_2\leq 1$, where $f_{k,(i,m-i)}$ is the $(i,m-i)$-homogeneous part of $f_k$. Let $c_1=1$ and $c_2=x_i$, where $x_i=i/(m+1)$ for $i=1,\cdot\cdot\cdot,m+1$. Let
\[
A=\left(
                  \begin{array}{cccc}
                    1 & x_1 & \cdot\cdot\cdot & x_1^m \\
                    1 & x_2 & \cdot\cdot\cdot & x_2^m \\
                    \cdot & \cdot &  &\cdot  \\
                    \cdot & \cdot &  &\cdot \\
                    \cdot & \cdot &  &\cdot  \\
                    1 & x_{m+1} & \cdot\cdot\cdot & x_{m+1}^m \\
                  \end{array}
                \right)\mbox{ and }
y=\left(
    \begin{array}{c}
      f_{k,(m,0)}(w_1\oplus w_2) \\
      f_{k,(m-1,1)}(w_1\oplus w_2) \\
      \cdot \\
      \cdot \\
      \cdot \\
      f_{k,(0,m)}(w_1\oplus w_2). \\
    \end{array}
  \right)
\]
By \eqref{e6.15}, we get
\begin{equation}
|Ay|\leq C.
\end{equation}
Note that
\[\det(A)=\prod_{1\leq j<i\leq m+1}(x_i-x_j)\neq 0\]
and therefore $A$ is invertible. Then we have $\|f_{k,(\alpha_1,m-\alpha_1)}\|\leq C$.

Similarly, we can prove $\|f_{k,(\alpha_1,\alpha_2, m-\alpha_1-\alpha_2)}\|\leq C$ if we decompose $E'_1$ by $E'_1=E_2\oplus E'_2$, where
$E_2=\mathrm{span}_{\mathbb{R}}\{v_2\}$ and $E'_2=\mathrm{span}_{\mathbb{R}}\{v_3,\cdots,v_n\}$.
Continuing in this way, we can decompose $\mathbb{R}^n$ by $\mathbb{R}^n=E_1\oplus\cdots \oplus E_n$ and prove that $\|f_{k,\alpha}\|\leq C$.
Thus, we have $\|\widehat{f_{k,\alpha}}\|\leq C$ and $\|\widehat{f}\|\leq C$. This completes the proof of \underline{{\textbf {Claim}}}.

(ii) The proof is similar to that of Lemma 5.6 in \cite{Tam06}.

\end{proof}

Let $\beta$ be the constant in Lemma \ref{l6.11}. Note that $\beta$ does not depend on $i$ and $m$ provided $m\geq m_0$, where $m_0\geq 2$ is the integer in Lemma \ref{l6.8}. Fix $m\geq m_0$ such that
\begin{equation}
\delta^m\leq \frac{1}{2}\beta^{-1}
\end{equation}
where $0<\delta<1$ is the constant in \eqref{e6.11}. Let $G_{i,m}$, $T_{i,m}$ be the maps given in Lemma \ref{l6.9} which are defined on $D(r_m)$, $0<r_m<r_1<1$.
For simplicity, we denote $G_{i,m}$ by $G_i$, $T_{i,m}$ by $T_i$ and $r_m$ by $r$.
\begin{lemma}\label{l6.12}
Let $k\geq 0$ be an integer. Then
\[\Psi_k=\underset{l\rightarrow \infty}{\lim}G_{k+1}^{-1}\circ G_{k+2}^{-1}\circ\cdot\cdot\cdot\circ G_{k+l}^{-1}\circ T_{k+l}\circ F_{k+l}\circ\cdot\cdot\cdot\circ F_{k+2}\circ F_{k+1}\]
exists and is a nondegenerate smooth map from $D(r)$ to $\mathbb{R}^n$. Moreover, there is a constant $\gamma>0$ which is independent of $k$ such that
\begin{equation}
\gamma^{-1}D(r)\subset \Psi(D(r))\subset \gamma D(r).
\end{equation}
\end{lemma}
\begin{proof}
The proof is similar to that of Lemma 5.7 in \cite{Tam06}.
\end{proof}

Now we are ready to prove Theorem \ref{uniformization}.

\begin{proof}[Proof of Theorem \ref{uniformization}]
Since the flow \eqref{e1.2} is solvable, we have the corresponding K\"{a}hler-Ricci flow \eqref{e6.2} on the tangent bundle $TM$
is also solvable. Because the initial metric $g^T_0$ has maximum volume growth, we can construct $\Phi_t$ and $F_i$ as above.
We can also construct $G_i$, $T_i$ as in Lemma \ref{l6.9} so that Lemmas \ref{l6.11} and \ref{l6.12} are true.
Let $\Omega_i=\Phi_{iT}(D(r))$ where $r>0$ is the constant in Lemma \ref{l6.12}. By Proposition \ref{p6.3} and the fact that the solution
$g(t)$ of \eqref{e6.7} decays exponentially, $\{\Omega_i\}_{i\geq 1}$ exhausts $M$.
Consider the following smooth maps from $\Omega_i$ to $\mathbb{R}^n$:
\[S_i=G_1^{-1}\circ\cdots\circ G_i^{-1}\circ T_i\circ \Phi_{iT}^{-1}.\]
For each fixed $k$ and $l\geq 1$
\[S_{k+l}=G_1^{-1}\circ\cdots\circ G_k^{-1}\circ [G_{k+1}^{-1}\circ\cdots\circ G_{k+l}^{-1}\circ T_{k+l}\circ F_{k+l}\circ\cdots\circ F_{k+1}]\circ \Phi_{kT}^{-1}.\]
By Lemma \ref{l6.12}, we have $S=\lim_{i\rightarrow \infty}S_i$ exists and is a nondegenerate smooth map from $M$ into $\mathbb{R}^n$.
Moreover, $S=G_1^{-1}\circ\cdots \circ G_k^{-1}\circ \Psi_k \circ \Phi_{kT}^{-1}$ on $\Omega_k$, where $\Phi_i$ is the nondegenerate smooth map
in Lemma \ref{l6.12}. Hence
\[S(\Omega_k)=G_1^{-1}\circ\cdots\circ G_k^{-1}\circ \Psi_k(D(r))\supset G_1^{-1}\circ\cdots\circ G_k^{-1}(\gamma^{-1}D(r))\]
by Lemma \ref{l6.12}, for some $\gamma$ independent of $k$. Therefore $S(M)=\mathbb{R}^n$ by Lemma \ref{l6.11} (ii). This completes the proof of
Theorem \ref{uniformization}.
\end{proof}



\end{document}